\documentclass{amsart}

\usepackage{amsmath}
\usepackage{mathtools}
\usepackage{amsthm,amssymb,color,comment}
\usepackage{bbm}
\usepackage{hyperref}
\usepackage[TS1,T1]{fontenc}
\usepackage{inputenc}
\usepackage{dsfont}
\usepackage{tikz}
\usepackage{enumitem}
\usepackage{cite}

\numberwithin{equation}{section}
\numberwithin{subsection}{section}

\theoremstyle{plain}
\newtheorem{thm}{Theorem}[section]
\newtheorem{cor}[thm]{Corollary}
\newtheorem{lem}[thm]{Lemma}
\newtheorem{prop}[thm]{Proposition}

\theoremstyle{definition}
\newtheorem{defn}[thm]{Definition}
\theoremstyle{remark}

\newcommand{\norm}[1]{\left\Vert#1\right\Vert}
\newcommand{\abs}[1]{\left\vert#1\right\vert}
\newcommand{\set}[1]{\left\{#1\right\}}
\newcommand{\R}{\mathbb R}

\newcommand{\N}{\mathbb N}

\newcommand{\C}{\mathbb C}

\newcommand{\Z}{\mathbb Z}
\newcommand{\eps}{\varepsilon}

\newcommand{\diverg}{\operatorname{div}}

\newcommand{\curl}{\operatorname{curl}}
\title[Existence of Strong Solutions and Convergence of Vorticity]{Strong Solutions and Applications to the Inviscid Limit for 2D Navier--Stokes with Navier Slip}

\author{Josef Demmel}
\address{Department of Mathematics, Friedrich-Alexander-Universit\"at Erlangen-N\"urnberg, Cauerstr.~11, 91058 Erlangen, Germany}
\email{josef.demmel@fau.de}
\author{Emil Wiedemann}
\address{Department of Mathematics, Friedrich-Alexander-Universit\"at Erlangen-N\"urnberg, Cauerstr.~11, 91058 Erlangen, Germany}
\email{emil.wiedemann@fau.de}
\date{\today}

\keywords{Navier--Stokes equations, Navier boundary conditions, strong solutions, vanishing viscosity}
\subjclass[2020]{Primary 76D03; Secondary 35Q30, 76D09.}

\begin{document}

\begin{abstract}
    We analyze the two-dimensional incompressible Navier--Stokes equations on a smooth, bounded, simply connected domain with Navier boundary conditions with friction coefficient $\alpha\in C^2$. For initial vorticity in $L^2$, we show that the unique weak solution for velocity is, in fact, strong and satisfies the Navier slip conditions for any positive time. The key idea is to consider a shifted vorticity, which vanishes on the boundary, and to study the Laplacian subject to Navier boundary conditions. We prove that this boundary-value problem is elliptic in the sense of Agmon--Douglis--Nirenberg. As an application of this scheme, we establish uniform-in-time strong convergence of the vorticity in the vanishing viscosity limit for initial vorticity in $L^p$ with $p>2$. We utilize a purely interior framework from Seis, Wiedemann, and Wo\'{z}nicki, originally derived for no-slip, and upgrade local to global convergence. Moreover, we show that the total bulk viscous enstrophy dissipation vanishes.
\end{abstract}

\maketitle

\section{Introduction}
Let $\Omega \subset \R^{2}$ be a bounded, simply connected domain whose boundary $\partial\Omega$ is of class $C^\infty$. We consider the two-dimensional, homogeneous, incompressible Navier--Stokes system with Navier boundary conditions (also called Navier friction conditions)
\begin{equation}\label{navier-stokes}
   \left\{
	\begin{aligned}
		&\partial_{t}u^{\nu} + (u^{\nu}\cdot\nabla)u^{\nu} + \nabla p^{\nu} = \nu\Delta u^{\nu}\text{ in }(0,T)\times\Omega, \\
		&\diverg u^{\nu} = 0\text{ in }(0,T)\times\Omega,\\
        &u^{\nu}\cdot n = 0\text{ on }(0,T)\times\partial\Omega,\\
        &2(Du^{\nu})_{S}n\cdot\tau + \alpha u^{\nu}\cdot \tau = 0 \text{ on }(0,T)\times\partial\Omega,
	\end{aligned}\right.
\end{equation}
where $u^{\nu}$ is the fluid velocity, $p^{\nu}$ is the pressure, $\alpha\in C^2(\partial\Omega)$ is the friction coefficient, $n$ and $\tau$ are the unit normal and tangent vector fields, $\nu>0$ is the viscosity, and $(Dv)_{S}$ denotes the rate-of-strain tensor,
\begin{equation*}
    (Dv)_{S} = \frac{1}{2}\left(\nabla v+\nabla v^{T}\right).
\end{equation*}
The Navier boundary condition, originally introduced by C. Navier in 1827 \cite{navier}, allows the tangential ``slip'' velocity to be nonzero. It can thus be viewed as a relaxation of the usual Dirichlet, or ``no-slip'' boundary condition ($u^\nu=0$), since formally $\alpha=\infty$ recovers $u\cdot\tau=0$ on the boundary. The choice of appropriate boundary condition depends on the physical application at hand. In geophysical fluid dynamics, for example, Navier boundary conditions have been physically justified (see \cite{pedlosky}). For an overview of experimental results about the occurrence of slip for various models, we refer the reader to \cite{lauga}.

We will assume, as in \cite{clopeau, nussenzveig}, that the friction coefficient $\alpha$ is independent of viscosity $\nu$. It is well known that under these circumstances there exists a unique weak Leray--Hopf solution $u^{\nu}$, provided $u^{\nu}(0,\cdot)\in L^2(\Omega)$ is divergence-free and tangent to the boundary in an appropriate weak sense (more details in Section 3.1).

In contrast to the Dirichlet boundary condition, the Navier boundary condition admits a formulation in terms of vorticity thanks to the differential equality
\begin{equation}\label{navier-vorticity}
    (Du)_Sn\cdot\tau+\kappa(u\cdot\tau)=\frac{1}{2}\curl u \text{ on }\partial\Omega,
\end{equation}
where $u$ is a vector field with zero normal component and $\kappa$ is the curvature of $\partial\Omega$. (For more details, we refer to \cite[Lemma 2.1]{clopeau}.) Combined with taking the curl of the momentum equation in \eqref{navier-stokes}, we obtain the vorticity formulation of the Navier--Stokes system with Navier boundary condition
\begin{equation}\label{navier-stokes-vorticity}
    \left\{
	\begin{aligned}
		&\partial_{t}\omega^{\nu} + u^{\nu}\cdot\nabla \omega^{\nu} = \nu\Delta \omega^{\nu}\text{ in }(0,T)\times\Omega, \\
		&\omega^{\nu}=\curl u^{\nu}\text{ in }(0,T)\times\Omega,\\
        &u^{\nu}\cdot n = 0\text{ on }(0,T)\times\partial\Omega,\\
        &\omega^{\nu}=(2\kappa-\alpha)u^{\nu}\cdot\tau\text{ on }(0,T)\times\partial\Omega.
	\end{aligned}\right.
\end{equation}

For a fixed time $t$, the velocity $u^{\nu}$ can be recovered from the vorticity $\omega^{\nu}$ using the Biot--Savart law. Explicitly, we write
\begin{equation*}
    u^{\nu}=K_{\Omega}(\omega^{\nu}),
\end{equation*}
where $K_{\Omega}$ is an integral operator with kernel $\nabla^{\perp}G_{\Omega}$ (the rotated gradient of the Dirichlet Green's function), hence $K_\Omega$ has order $-1$. We will frequently use the Biot--Savart law to get an initial velocity for a given initial vorticity.\\
Our first result concerns the existence of strong solutions to \eqref{navier-stokes}, by which we mean weak solutions in $C([0,T];H^{1}(\Omega))\cap L^{2}(0,T;H^{2}(\Omega))$. Continuity in $H^1(\Omega)$ requires the initial velocity to belong to $H^1(\Omega)$ and to satisfy the trace compatibility associated with the boundary condition. For the Navier condition, this amounts to tangency to the boundary. The subsequent theorem establishes the existence of strong solutions in exactly that class of divergence free initial data.
\begin{thm}\label{theorem-strong-solutions}
    Fix $\nu>0$. Let $\Omega\subset\R^2$ be a bounded, simply connected domain with $C^\infty$ boundary, $\alpha\in C^{2}(\partial\Omega)$, $T\in(0,\infty)$ and $\omega_{0}^{\nu}\in L^{2}(\Omega)$. Then the unique solution $u^{\nu}$ of \eqref{navier-stokes} for $u^{\nu}(0,\cdot)=K_{\Omega}(\omega_{0}^{\nu})$ is a strong solution in the sense that:
    \begin{enumerate}
        \item[(i)] $u^\nu\in L^2(0,T;H^1(\Omega))$, divergence-free and tangent to the boundary, is a weak solution to \eqref{navier-stokes} for $u^{\nu}(0,\cdot)=K_{\Omega}(\omega_{0}^{\nu})$, i.e.,
        \begin{equation*}
        \frac{d}{dt}\int_{\Omega}u^\nu\cdot v dx+\int_{\Omega}(u^\nu\cdot\nabla u^\nu)\cdot vdx+\nu\int_{\Omega}\nabla u^\nu:\nabla vdx=\nu\int_{\partial\Omega}(\kappa-\alpha)u^\nu\cdot v dS,
    \end{equation*}
    for all $v\in H^1(\Omega)$, divergence-free and tangent to the boundary, and $u^\nu(0,\cdot)=K_{\Omega}(\omega_{0}^{\nu})$.
    \item[(ii)] $u^{\nu}\in C([0,T];H^{1}(\Omega))\cap L^{2}(0,T;H^{2}(\Omega))$ and $\partial_tu^{\nu}\in L^{2}(0,T;L^{2}(\Omega))$.
    \end{enumerate}
    Moreover, the strong solution satisfies the Navier boundary condition almost everywhere, i.e., 
    \begin{equation*}
        2(Du^{\nu}(t,x))_{S}n(x)\cdot\tau(x) + \alpha(x) u^{\nu}(t,x)\cdot \tau(x) = 0 \text{ for a.e. }(t,x)\in (0,T)\times\partial\Omega.
    \end{equation*}
    Here, ``almost every $x\in\partial\Omega$'' is intended with respect to the one-dimensional Hausdorff measure on $\partial\Omega$.
\end{thm}
Strong solutions possess numerous mathematically desirable properties and hence play a fundamental role in the study of the Navier--Stokes equations. In the present setting, they provide for instance sufficient regularity compared to the weak solution (see Definition \ref{weak-solution-velocity}) to interpret the Navier boundary condition as an actual boundary trace equality. 

For the Dirichlet boundary condition, the existence of strong solutions with divergence free initial data in $H^1_0(\Omega)$ is a long-known classical result (see, for example, \cite[Theorem 3.10]{temam}). For the Navier boundary condition, the classical result of Clopeau, Mikeli\'{c}, and Robert \cite{clopeau} gives existence of strong solutions for divergence free initial velocity in $H^2(\Omega)$ that satisfies the Navier boundary condition and whose vorticity is bounded. This class of initial data is also referred to as ``compatible'' (originally in \cite{nussenzveig}). It is worth noting that the strong solution in \cite{clopeau} enjoys some additional properties. The velocity $u^\nu$ is in $C([0,T];H^{2}(\Omega))$ (also $\partial_t u^\nu\in L^2(0,T;H^1(\Omega))$) and the vorticity $\omega^\nu$ is bounded in space-time. In the special case of constant friction $\alpha\geq0$, this result also follows from a recent work of Shao, Simonett, and Wilke \cite{wilke}, who work in the broader geometric framework of compact Riemannian manifolds with boundary.\\ 

\textbf{Strategy for Theorem \ref{theorem-strong-solutions}}\;
The strategy resembles the classical proof for the Dirichlet case in \cite[Theorem 3.10]{temam}. We choose a suitable approximating sequence and establish the additional regularity through uniform estimates. With weak compactness, the limit inherits the desired regularity. For the approximating sequence, we actually pick the strong solutions of Clopeau et al. This is possible as Lemma \ref{compatible-lemma} (originally \cite[Lemma 1]{nussenzveig}) provides the approximation of the initial data by compatible functions and has also been done in \cite{nussenzveig}. The uniform bounds are derived by considering the enstrophy balance for a ``shifted'' vorticity 
\begin{equation*}
    \overline{\omega^\nu}:=\omega^\nu-(2\kappa-\alpha)u^\nu\cdot\tau,
\end{equation*}
whose boundary trace vanishes. Roughly speaking, we trade with this shift the boundary integral
\begin{equation*}
    \nu\int_{\partial\Omega}\omega^\nu\partial_n\omega^\nu dS,
\end{equation*}
for an additional bulk integral. The problem with the boundary integral is that we lack information about $\partial_n\omega^\nu$ (see also \cite[Remark 2]{nussenzveig}), but we can control the bulk integral by a pressure estimate (Lemma \ref{pressure-lemma}). We can establish a uniform estimate of $\omega^\nu$ in $L^2(0,T;H^1(\Omega))$ and hence a bound for the Laplacian with the equality $\nabla^\perp\omega^\nu=\Delta u^\nu$. The remaining difficulty is to recover $H^2(\Omega)$ control of $u^\nu$. Unlike in the Dirichlet case, the elliptic estimate $\norm{u^\nu}_{H^2(\Omega)}\lesssim\norm{\Delta u^\nu}_{L^2(\Omega)}+\norm{ u^\nu}_{L^2(\Omega)}$ is not trivial, since the Navier boundary conditions produce a mixed boundary value problem for the vector Laplacian. The Agmon--Douglis--Nirenberg theory \cite{agmon} gives us the desired estimate, provided that the Laplacian with the Navier boundary condition is a \textit{well-behaved} elliptic operator in the sense of Definition \ref{elliptic-problem}, which we will verify. For the convenience of the reader, we restate the relevant Agmon--Douglis--Nirenberg theory in Appendix B.\\

A further application of the preceding scheme is the study of the vanishing viscosity limit, i.e., $\nu\to0$, for the vorticity $\omega^\nu$. For $\nu=0$, the Navier--Stokes system \eqref{navier-stokes-vorticity} formally reduces to the Euler system in vorticity formulation 
\begin{equation}\label{euler-vorticity}
    \left\{
	\begin{aligned}
		&\partial_{t}\omega + u\cdot\nabla \omega = 0 \text{ in }(0,T)\times\Omega,\\
		&\omega=\curl u \text{ in }(0,T)\times\Omega,\\
        &u\cdot n=0 \text{ on }(0,T)\times\partial\Omega.
	\end{aligned}\right.
\end{equation}
With the aforementioned approximation sequence, we can apply the local result of Seis, Wiedemann and Wo\'{z}nicki \cite[Theorem 1.3]{wiedemann} and upgrade to strong global convergence for vorticity $\omega^\nu$ in $C([0,T];L^p(\Omega))$. We also establish that the total bulk viscous enstrophy dissipation vanishes in the limit.
\begin{prop}\label{strong-convergence-vorticity}
           Let $\Omega\subset\R^2$ be a bounded, simply connected domain with $C^\infty$ boundary, $\alpha\in C^{2}(\partial\Omega)$, $T\in(0,\infty)$ and $\{\omega_{0}^{\nu}\}_{\nu>0}\subset L^{p}(\Omega)$ for $p>2$ such that
           \begin{equation*}
               \omega_{0}^{\nu}\to \omega_{0} \text{ strongly in }L^{p}(\Omega) \text{ as }\nu\to0,
           \end{equation*}
           for some $\omega_{0}\in L^{p}(\Omega)$ and let $u^{\nu}$ be the unique solution to \eqref{navier-stokes} for $u^{\nu}(0,\cdot)=K_{\Omega}(\omega_{0}^{\nu})$. Then, for the associated vorticity $\omega^{\nu}=\curl u^{\nu}$, there exists a sequence $\nu_{k}\to 0$ such that 
           \begin{equation*}
               \omega^{\nu_{k}}\to \omega \text{ strongly in } C([0,T];L^{p}(\Omega))\text{ as }k\to\infty,\text{ }
           \end{equation*}
           and $\omega$ is a weak solution of \eqref{euler-vorticity} for $\omega(0,\cdot)=\omega_{0}$. 
           Moreover, it holds that
           \begin{equation*}
               \lim_{k\to\infty}\nu_k\int_0^T\int_{\Omega}\abs{\nabla\omega^{\nu_k}(t,x)}^2dxdt=0.
           \end{equation*}
\end{prop}
The transition from Navier--Stokes to Euler has been an extensively researched topic for decades, but is still only partially understood, especially in the presence of physical boundaries (for an overview of past results, we refer to \cite{maekawa}). A major challenge is that the decrease in the order of the system as it approaches the limit leads to a discrepancy in the boundary conditions. In particular, for the Dirichlet boundary condition, the mismatch leads to boundary layer effects, which are still, to this day, not adequately understood. Although some results exist regarding convergence in the vanishing viscosity limit under additional conditions (see, for instance, \cite{bardos, constantin, kato, masmoudi, sammartino1, sammartino2, wiedemann, wang}), the issue is still largely unresolved. For vorticity, convergence in $L^p$ for $p>1$ can even be ruled out if the tangential velocity at the boundary of the limiting Euler solution is nonzero, i.e., a boundary layer forms in the inviscid limit\cite{kelliher2}.

The Navier boundary conditions are better understood regarding the vanishing viscosity limit, as the boundary effects are more manageable. The strong convergence of velocity was first established in \cite{clopeau} for bounded initial vorticity and later generalized in \cite{nussenzveig} for initial vorticity in $L^p$ with $p>2$. For the same class of initial data, strong convergence of vorticity in $L^r(0,T;L^p(\Omega))$ for finite $1<r<\infty$ was shown in \cite{chemetov} by a different, entropy-based argument. The result applies to general nonhomogeneous Navier boundary conditions, which contains the present setting as a special case.\\

\textbf{Strategy for Proposition \ref{strong-convergence-vorticity}}\; The proof essentially hinges on two past results. Seis et al. showed in \cite[Theorem 1.3]{wiedemann} a local version of Proposition \ref{strong-convergence-vorticity}, i.e., convergence in \\$C_{loc}([0,T);L^{q}_{loc}(\Omega))$ for $q\in[1,p)$, for the no-slip boundary condition, under an additional hypothesis.  They assume that the $p$-enstrophies are locally bounded uniformly in viscosity, that is, for every compact set $K\subset\Omega$, there exists a constant $C_K$ such that
\begin{equation}\label{local-p-enstrophies}
    \sup_{\nu\in(0,1)}\sup_{t\in(0,T)}\int_K\abs{\omega^\nu(t,x)}^pdx\leq C_K.
\end{equation}
With Navier boundary conditions, there is no need for an additional hypothesis. In \cite{nussenzveig}, it was shown that for every $p>2$, the $p$-enstrophies are even globally bounded uniformly in time and viscosity
\begin{equation}\label{global-p-enstrophies}
    \sup_{\nu\in(0,1)}\sup_{t\in(0,T)}\int_\Omega\abs{\omega^\nu(t,x)}^pdx\leq C\left(\norm{\omega_0^\nu}_{L^p(\Omega)}\right).
\end{equation}
Provided that the statement in \cite{wiedemann} also holds for the Navier boundary condition, we could upgrade to convergence in $C([0,T];L^{q}(\Omega))$, which is almost Proposition \ref{strong-convergence-vorticity}, due to the global bound \eqref{global-p-enstrophies}. At first glance, it should be no problem to adapt the proof in \cite{wiedemann} to the Navier boundary condition, as the argument is purely local. But the previously known regularity in our scenario poses a problem for \cite[Lemma 3.3]{wiedemann}, as integration by parts may not be rigorously justified. The lemma states that up to an additive constant, which vanishes as $\nu\to0$, the solution $\omega^\nu$ is a renormalized subsolution of the transport equation (see Lemma \ref{wiedemann-lemma}). To overcome this obstacle, we use the same approximation scheme as before. By a renormalization property (Lemma \ref{convex-renormalization-property}), we can even upgrade to the endpoint $p=q$, as we can rule out concentrations, which are the only possible defect. 

As a final remark, we want to point out that our assumption $\partial\Omega\in C^\infty$ can easily be weakened at least to $C^4$, but we chose to work with a smooth boundary for simplicity. 
\section{Notation}
In the subsequent discussion, $x,y\in\overline{\Omega}$ will always be spatial variables, while $t,s\in[0,T]$ are reserved for time variables. Between vectors $v,w\in\R^d$, 
we will use $v\cdot w$ for the standard scalar product and $A:B$ for the scalar product between matrices $A,B\in \R^{d\times d}$. The identity matrix is denoted by $I_d\in\R^{d\times d}$. For $h,k\in\R^2\setminus\{0\}$, the directional derivative along $h$ is defined by
 \begin{equation*}
     \partial_h v:= (h\cdot\nabla)v,
 \end{equation*}
 and the outer product $\otimes$ is given by
 \begin{equation*}
     h\otimes k := \left(\begin{array}{cc}
         h_1 k_1  & h_1 k_2\\
         h_2 k_1  & h_2 k_2
    \end{array}\right).
 \end{equation*}
Note that we have, for vector-valued $v$, the identities $k\otimes h : \nabla v = \partial_h v\cdot k=k\cdot(\nabla vh).$
In two dimensions, the curl operator is defined as
\begin{equation*}
    \curl v:=\partial_{x_1}v_2-\partial_{x_2}v_1.
\end{equation*}
For a scalar $\psi$, we introduce the orthogonal gradient as 
\begin{equation*}
    \nabla^{\perp}\psi:=\left(\begin{array}{c}
         -\partial_{x_{2}}  \\ \partial_{x_{1}}
    \end{array}\right)\psi.
\end{equation*}
Following \cite{clopeau}, we introduce the following Hilbert spaces
\begin{equation*}
    \begin{aligned}
        &L^2_0(\Omega)=\set{v\in L^2(\Omega):\int_\Omega v\;dx=0},\\
        &H=\set{v\in L^{2}(\Omega): \diverg v=0 \text{ in } \Omega \text{ and }v\cdot n = 0 \text{ on }\partial\Omega},\\
        &V=\set{v\in H^{1}(\Omega): \diverg v=0 \text{ in } \Omega \text{ and }v\cdot n = 0 \text{ on }\partial\Omega},\\
        &W=\set{v\in H^{2}(\Omega)\cap V: 2(Dv)_{S}n\cdot\tau + \alpha v\cdot\tau=0 \text{ on }\partial\Omega}.\\
    \end{aligned}
\end{equation*}
Note that the boundary conditions in these definitions are intended in an appropriate trace sense. In line with \cite{wiedemann}, convergence of $f_n$ to $f$ in $C_{loc}([0,T);L^{q}_{loc}(\Omega))$ means that for any $\chi\in C^\infty_c(\Omega)$ and $T'\in(0,T)$ it holds that $f_n\chi\to f\chi$ in $C([0,T'];L^{q}(\Omega))$. With $\mathbb{P}$, we denote the Leray projection.
In terms of notation, we do not distinguish between scalar and vector-valued function spaces.
Lastly, we follow the convention that the value of constants in an inequality can vary from line to line and can depend on the domain and functions of the domain. If we want  to express explicitly that the constant depends on some parameters $a_1,..,a_n$, we write $C(a_1,..,a_n)$.
\section{Viscous solutions}
In this part, we want to fix $\nu>0$ and consider solutions to \eqref{navier-stokes} so we omit the superscript $\nu$. Before recalling some facts from \cite{clopeau, kelliher, nussenzveig}, we rewrite the Navier boundary condition, which will later be useful for Lemma \ref{riesz-estimate}.
\begin{lem}\label{boundar-identity}
    Let $v\in W$. Then we have
    \begin{equation}\label{elliptic-boundary}
        ((n\cdot\nabla)v)\cdot \tau + (\alpha-\kappa)(v\cdot\tau)=0 \text{ on }\partial\Omega.
    \end{equation}
\end{lem}
\begin{proof}
    Due to $v\cdot n=0$ being constant on $\partial\Omega$, differentiating along $\tau$ gives
    \begin{equation*}
        0 = \partial_\tau(v\cdot n)= \partial_\tau v\cdot n + v\cdot\partial_\tau n = \partial_\tau v\cdot n + (v\cdot\tau)\tau\cdot\partial_\tau n = \partial_\tau v\cdot n + \kappa (v\cdot\tau).
    \end{equation*}
    The Navier boundary condition, on the other hand, gives
		\begin{equation*}
		-\alpha v\cdot \tau=2(Dv)_{S}n\cdot\tau=(\nabla v n)\cdot\tau+((\nabla v)^T n)\cdot \tau=\partial_nv\cdot\tau+\partial_\tau v\cdot n.
		\end{equation*}
		Combining both and rearranging, we arrive at~\eqref{elliptic-boundary}.
\end{proof}
\subsection{Weak solutions}
In the formal derivation of the distributional formulation of the Navier--Stokes equations with Navier boundary condition, one employs~\eqref{elliptic-boundary} to obtain
\begin{defn}\label{weak-solution-velocity}
    For $\nu>0$ and $u_{0}\in H$, $u\in C([0,T];H)\cap L^2(0,T;V)$ is called a weak solution to \eqref{navier-stokes} if $u(0)=u_{0}$ and 
\begin{equation}\label{weak-formulation}
    \frac{d}{dt}\int_{\Omega}u\cdot v dx+\int_{\Omega}(u\cdot\nabla u)\cdot vdx+\nu\int_{\Omega}\nabla u:\nabla vdx=\nu\int_{\partial\Omega}(\kappa-\alpha)u\cdot v dS,
\end{equation}
for all $v\in V$, in the sense of distributions.
\end{defn}
Like in the Dirichlet case, there exists a unique weak (Leray--Hopf) solution for $u_{0}\in H$.
\begin{thm}[Theorem 6.1, \cite{kelliher}]\label{energybound}
    For $\nu>0$ and $u_{0}\in H$, there exists a unique weak solution $u$ to \eqref{navier-stokes} with $u(0,\cdot)=u_0$. Moreover, $\partial_t u\in L^2(0,T;V')$ and we have the energy inequality
    \begin{equation}\label{energyestimate}
        \norm{u(t)}_{L^2(\Omega)}\leq e^{C(\alpha)\nu t}\norm{u_0}_{L^2(\Omega)},
    \end{equation}
    with $C(\alpha)=0$ if $\alpha\geq0$ on $\partial\Omega$.
\end{thm}
In contrast to the Dirichlet boundary condition, Lopes Filho et al.\ showed in \cite{nussenzveig} that given any initial vorticity $\omega_0\in L^p(\Omega)$ for $p>2$, there exists a uniform $L^p$-bound for the vorticity. (They assume $\alpha\geq0$, but it is generalized in \cite[Theorem 7.3]{kelliher} to indefinite $\alpha$.)
\begin{prop}[Proposition 1, \cite{nussenzveig}]
    Fix $0<\nu<1$. Let $\omega_0\in L^p(\Omega)$ for some $p>2$ and $u$ be the unique solution of \eqref{navier-stokes} with $u(0,\cdot)=K_{\Omega}(\omega_{0})$. Then we have for the associated vorticity $\omega=\curl u$ the uniform estimate
    \begin{equation}\label{uniform-vorticity-bound}
    \norm{\omega}_{L^{\infty}(0,T;L^{p}(\Omega))}\leq C\left(\norm{\omega_{0}}_{L^{p}(\Omega)}\right),
\end{equation}
with constant $C>0$ independent of viscosity $\nu$.
\end{prop}
As a straightforward consequence, this also gives rise to a uniform bound for the velocity $u$, which we state in the following corollary.
\begin{cor}\label{uniform-velocity-corollary}
    Fix $0<\nu<1$ and let $\omega_0\in L^p(\Omega)$ for some $p>2$. Then we have for the unique solution $u$ of \eqref{navier-stokes} with $u(0,\cdot)=K_{\Omega}(\omega_{0})$ the uniform estimate
    \begin{equation}\label{uniform-velocity-bound}
        \norm{u}_{L^{\infty}(0,T;W^{1,p}(\Omega))}\leq C\left(\norm{\omega_{0}}_{L^{p}(\Omega)}\right),
    \end{equation}
    with constant $C>0$ independent of viscosity $\nu$. Moreover, we also have a uniform bound for $u$ in $L^{\infty}(0,T;C(\overline{\Omega}))$.
\end{cor}
\begin{proof}
        First, note that the only constant vector field in $V$ is 0, so the Poincar\'e inequality applies. Thus, it suffices to bound the $L^{p}$-norm of $\nabla u$. With the Calder\'on--Zygmund inequality (see Lemma \ref{curl_estimate}) and \eqref{uniform-vorticity-bound}, we deduce that
    \begin{equation*}
        \begin{aligned}
            \norm{\nabla u}_{L^{p}(\Omega)} \lesssim   \norm{\omega}_{L^{p}(\Omega)} \lesssim  \norm{\omega_{0}}_{L^{p}(\Omega)}.
        \end{aligned}
    \end{equation*}
    The moreover part is a consequence of Morrey's inequality.
\end{proof}
An important idea is to approximate the initial vorticity by so-called compatible functions in order to work with strong solutions.
\begin{lem}[Lemma 1, \cite{nussenzveig}]\label{compatible-lemma}
    Let $\omega\in L^p(\Omega)$ for $p>1$. Then there exists a sequence of compatible functions $\set{\omega_n}$, i.e., $\{\omega_{n}\}\subset H^{1}(\Omega)\cap L^{\infty}(\Omega)$ with $u_{n}=K_{\Omega}(\omega_{n})\in W$, that converges strongly to $\omega$ in $L^p(\Omega)$.
\end{lem}
\subsection{Compatible initial data}
For initial data which actually satisfies the Navier boundary condition in a trace sense, Clopeau et al. \cite{clopeau} showed existence of a strong solution. We mean by a strong solution that it is a weak solution and belongs at least to $C([0,T];H^1(\Omega))\cap L^2(0,T;H^2(\Omega))$.
\begin{thm}[Theorem 2.3, \cite{clopeau}]\label{theorem-clopeau}
    For $\nu>0$ and $u_{0}\in W$, there exists a unique strong solution $u\in C([0,T];W)$ to \eqref{navier-stokes} with $u(0,\cdot)=u_0$. Moreover, $\partial_{t}u\in L^{2}(0,T;V)\cap C([0,T];H)$, the associated vorticity $\omega=\curl u$ is bounded, $\omega\in C([0,T];H^{1}(\Omega))\cap L^{\infty}((0,T)\times\Omega)$ and there exists a unique pressure $p\in C([0,T];H^1(\Omega)\cap L^2_0(\Omega))$ such that the momentum equation of \eqref{navier-stokes} holds a.e.\ on $(0,T)\times\Omega$.
\end{thm}
For both results, Theorem \ref{theorem-strong-solutions} and Proposition \ref{strong-convergence-vorticity}, we utilize a similar strategy. We approximate $\omega_{0}\in L^{p}(\Omega)$ by a sequence of compatible functions $\omega_{0,n}$ and consider the associated strong solutions $u_n$. To show Theorem \ref{theorem-strong-solutions}, we derive the corresponding uniform bound for $u_n$ (Lemma \ref{riesz-estimate}, Lemma \ref{pressure-lemma} and Proposition \ref{approximate-convergence}) and derive the additional regularity for $u$ by uniqueness of weak limits. For Proposition \ref{strong-convergence-vorticity}, we need to show the analogue of \cite[Lemma 3.3]{wiedemann}. First, we show it for $\omega_n=\curl u_n$ by improving the regularity of $\omega_n$ (Lemma \ref{additional-regularity}). To prove that it also holds for the limit $\omega$, we show strong convergence of $\omega_n$ to $\omega$ (Proposition \ref{approximate-convergence}).
\begin{lem}\label{riesz-estimate}
    Fix $\nu>0$ and let $\omega_{0}\in H^{1}(\Omega)\cap L^{\infty}(\Omega)$ with $u_{0}=K_{\Omega}(\omega_{0})\in W$. Then we have for the unique solution $u$ of \eqref{navier-stokes} with $u(0,\cdot)=K_{\Omega}(\omega_{0})$ the elliptic $H^2$-estimate
    \begin{equation*}
        \norm{u}_{H^2(\Omega)}\leq C\left(\norm{\Delta u}_{L^2(\Omega)}+\norm{u}_{L^2(\Omega)}\right),
    \end{equation*}
    where the constant $C>0$ only depends on the domain $\Omega$ and on the friction coefficient $\alpha$ in $C^1(\partial\Omega)$.
\end{lem}
\begin{proof}
    The idea is to use Theorem \ref{elliptic-theorem}, so we need to check whether the Laplacian with Navier boundary conditions is an elliptic boundary value problem in the sense of Definition \ref{elliptic-problem}. First, we need to write it as a boundary value problem of the form \eqref{elliptic-system}. The Laplacian corresponds to the elliptic operator 
    \begin{equation*}
        L(x,D)=I_2\left(D^{(2,0)}+D^{(0,2)}\right)
    \end{equation*}
    and the boundary operator (see Lemma~\ref{boundar-identity})
    \begin{equation*}
        B(x,D)= \left(\begin{array}{c}
            n(x)^T \\
            (\alpha(x)-\kappa(x))\tau^T(x)
        \end{array}\right)D^{(0,0)}+\left(\begin{array}{c}
            0 \\
            n_1(x)\tau^T(x)
        \end{array}\right)D^{(1,0)}+\left(\begin{array}{c}
            0 \\
            n_2(x)\tau^T(x)
        \end{array}\right)D^{(0,1)}
    \end{equation*}
    with $g\equiv0$ encodes the Navier boundary conditions in the form of \eqref{elliptic-boundary}.\\ 
    Next, we verify each condition of Definition \ref{elliptic-problem}:
    \begin{enumerate}
        \item[(i)] For ADN-ellipticity, we choose the weights $s_1=s_2=0$ and $t_1=t_2=2$. Obviously, conditions ($\alpha$) and ($\beta$) are satisfied. For ($\gamma$), notice that the principal symbol coincides with the operator itself, $L^p(x,D)=L(x,D)$, and therefore we have that
        \begin{equation*}
            \det L^p(x,\xi)= (\xi_1^2+\xi_2^2)^2\neq 0 \text{ for }\xi\neq0.
        \end{equation*}
        \item[(ii)] The Laplace operator is obviously a uniformly elliptic operator of order $2m=42$, hence $m=2$.
        \item[(iii)] Let $\xi,\xi'\in\R^2$ be linearly independent and $\sigma\in\C$. The polynomial
        \begin{equation*}
            \sigma\mapsto\det L^p(x,\xi+\sigma\xi')=\left(\abs{\xi}^2+2\sigma\xi\cdot\xi'+\sigma^2\abs{\xi'}^2\right)^2
        \end{equation*}
        has two roots, counted multiplicity, with positive imaginary part. Indeed, this (double) root is
				\begin{equation}\label{root}
				\frac{-\xi\cdot\xi'+\sqrt{(\xi\cdot\xi')^2-|\xi|^2|\xi'|^2}}{|\xi'|^2},
				\end{equation}
				and the discriminant is strictly negative because the Cauchy--Schwarz inequality holds strictly thanks to the linear independence of $\xi$ and $\xi'$.
				Thus, the operator $L(x,D)$ is regular elliptic.
        \item[(iv)] We start by choosing the weights $r_1=-2$ and $r_2=-1$. The principal part of $B(x,D)$ is then given by
        \begin{equation*}
            B^p(x,D)=\left(\begin{array}{c}
            n(x)^T \\
            0
        \end{array}\right)D^{(0,0)}+\left(\begin{array}{c}
            0 \\
            n_1(x)\tau^T(x)
        \end{array}\right)D^{(1,0)}+\left(\begin{array}{c}
            0 \\
            n_2(x)\tau^T(x)
        \end{array}\right)D^{(0,1)},
        \end{equation*}
        and is therefore independent of $\alpha$. Next, fix a point $x\in\partial\Omega$ and let $\xi\in\R^2\backslash\set{0}$ be any vector orthogonal to $n(x)$. Recall that we need to check whether the rows of $B^p(x,\xi+\sigma n)L'(x,\xi+\sigma n)$ are linearly independent modulo $M^+(x,\xi,\sigma)$. The root with positive imaginary part of $L^p(x,\xi+\sigma n)$ is $\sigma=i\abs{\xi}$ with multiplicity 2 (this is~\eqref{root} with $\xi'=n$) and therefore we have
        \begin{equation*}
            M^+(x,\xi,\sigma)=(\sigma-i\abs{\xi})^2.
        \end{equation*}
        The term $\xi+\sigma n$ never vanishes, so the adjugate matrix $L'(x,\xi+\sigma n)$ is given by
        \begin{equation*}
            \begin{aligned}
                L'(x,\xi+\sigma n)=&\det L^p(x,\xi+\sigma n)(L^p(x,\xi+\sigma n))^{-1}\\
                =&(\abs{\xi}^2+\sigma^2)^2\frac{1}{\abs{\xi}^2+\sigma^2}I_2\\
                =&(\abs{\xi}^2+\sigma^2)I_2.
            \end{aligned}
        \end{equation*}
        We are now in a position to verify the complementing condition. The matrix product is
        \begin{equation*}
            \begin{aligned}
                &B^p(x,\xi+\sigma n)L'(x,\xi+\sigma n)\\
                =&(\abs{\xi}^2+\sigma^2)\left(\left(\begin{array}{c}
            n^T \\
            0
        \end{array}\right)+\left(\begin{array}{c}
            0 \\
            n_1\tau^T
        \end{array}\right)(\xi_1+\sigma n_1)+\left(\begin{array}{c}
            0 \\
            n_2\tau^T
        \end{array}\right)(\xi_2+\sigma n_2)\right)\\
        =&(\abs{\xi}^2+\sigma^2)\left(\begin{array}{c}
            n^T \\
            (\xi_1 n_1+\xi_2 n_2)\tau^T+\sigma(n_1^2+n_2^2)\tau^T
            
        \end{array}\right)\\
        =&(\sigma-i\abs{\xi})(\sigma+i\abs{\xi})\left(\begin{array}{c}
            n^T \\
            \sigma\tau^T
        \end{array}\right),\\
            \end{aligned}
        \end{equation*}
        so we need to check whether the rows $(\sigma-i\abs{\xi})(\sigma+i\abs{\xi})n$ and $(\sigma-i\abs{\xi})(\sigma+i\abs{\xi})\sigma\tau$ are linearly independent modulo $(\sigma-i\abs{\xi})^2$. Assume that $C_1,C_2\in\C$ satisfy 
        \begin{equation*}
            (\sigma-i\abs{\xi})(\sigma+i\abs{\xi})(C_1 n+C_2\sigma\tau)\equiv0 \text{ mod }(\sigma-i\abs{\xi})^2.
        \end{equation*}
        The particular choice $\sigma=i\abs{\xi}$ yields
        \begin{equation*}
            C_1 n+C_2i\abs{\xi}\tau =0.
        \end{equation*}
        Since $n$ and $\tau$ are linearly independent, this can only be the case if 
        \begin{equation*}
            C_1=C_2=0.
        \end{equation*}
        Hence the rows of $B^p(x,\xi+\sigma n)L'(x,\xi+\sigma n)$ are linearly independent modulo \\$M^+(x,\xi,\sigma)$, so $L$ and $B$ satisfy the complementing condition.
        \end{enumerate}
    We have shown that our boundary value problem is elliptic and we can therefore use Theorem \ref{elliptic-theorem} for $q=0$ (in our setting $t'=2$ and $r'=0$). Notice that the coefficients of $L$ are constant in $x$, thus sufficiently regular, and the coefficients of $B$ only depend on $n,\tau,\kappa$ and $\alpha$, which all belong at least to $C^2(\partial\Omega)$. Finally, estimate \eqref{elliptic-estimate} yields
    \begin{equation*}
        \norm{u}_{H^2(\Omega)}\leq C\left(\norm{\Delta u}_{L^2(\Omega)}+\norm{u}_{L^2(\Omega)}\right),
    \end{equation*}
    with constant $C>0$ depending only on $\Omega$ and $\alpha$ (in $C^1(\partial\Omega)$, as $r_2=-1$).
\end{proof}
The subsequent lemma is a simple consequence of elliptic regularity for the Dirichlet problem.
\begin{lem}\label{additional-regularity}
Fix $\nu>0$ and let $\omega_{0}\in H^{1}(\Omega)\cap L^{\infty}(\Omega)$ with $u_{0}=K_{\Omega}(\omega_{0})\in W$. Then we have for the associated vorticity $\omega = \curl u$ to the unique solution $u$ of \eqref{navier-stokes} with $u(0,\cdot)=u_{0}$ the additional regularity
    \begin{equation}
        \omega\in L^{2}(0,T;H^{2}(\Omega)).
    \end{equation}
\end{lem}
\begin{proof}
     From the already established regularity, we know that the vorticity equation, 
     \begin{equation*}
         \partial_t\omega+(u\cdot\nabla)\omega=\nu\Delta\omega, 
     \end{equation*}
     holds in $L^2(0,T;H^{-1}(\Omega))$ and the left hand-side is in $L^2(0,T;L^{2}(\Omega))$. The canonical injection $\iota:L^2(\Omega)\to H^{-1}(\Omega)$ is injective and thus we also have $\Delta\omega\in L^2(0,T;L^{2}(\Omega))$. Writing $\tau,\kappa$ and $\alpha$ for their smooth extensions to $\overline{\Omega}$, we have
     \begin{equation*}
         (2\kappa-\alpha)u\cdot\tau\in C([0,T];H^2(\Omega)).
     \end{equation*} With the boundary condition in~\eqref{navier-stokes-vorticity} for the vorticity $\omega$, we thus have 
     \begin{equation}\label{vortexboundary}
         \omega-(2\kappa-\alpha)u\cdot\tau\in C([0,T];H^1_0(\Omega)).
     \end{equation}
     We can therefore use an elliptic regularity result such as \cite[Theorem 8.12]{gilbarg} to obtain $\omega\in L^{2}(0,T;H^{2}(\Omega))$, which completes the proof.
\end{proof}
For completeness, we give the following standard estimate for the pressure.
\begin{lem}\label{pressure-lemma}
    Fix $\nu>0$ and let $\omega_{0}\in H^{1}(\Omega)\cap L^{\infty}(\Omega)$ with $u_{0}=K_{\Omega}(\omega_{0})\in W$. Then we have for the associated pressure $p\in C([0,T];H^1(\Omega)\cap L^2_0(\Omega))$ to the unique solution $u$ of \eqref{navier-stokes} with $u(0,\cdot)=u_0$ the following estimate for any $t\in[0,T]$ and $\phi\in H^1_0(\Omega)$:
    \begin{equation}\label{pressure-estimate}
        \abs{\int_{\Omega}\nabla p(t)\cdot\phi dx}\leq C\norm{u(t)}^2_{L^4(\Omega)}\norm{\phi}_{H^1(\Omega)}+C\nu\norm{\nabla\omega(t)}_{L^2(\Omega)}\norm{\phi}_{L^2(\Omega)},   
    \end{equation}
    where $C$ depends only on the domain $\Omega$. 
\end{lem}
\begin{proof}
    Fix an arbitrary time $t\in[0,T]$ and $\phi\in H^1_0(\Omega)$. By Lax--Milgram, there exists a unique solution $\psi\in H^1(\Omega)\cap L^2_0(\Omega)$ to the weak Neumann problem
    \begin{equation}\label{weak-neumann}
        \int_\Omega \nabla\psi\cdot\nabla\xi dx = -\int_\Omega \phi\cdot\nabla\xi dx,\quad\xi\in H^1(\Omega).
    \end{equation}
    With $\xi=\psi$, we obtain 
    \begin{equation*}
        \norm{\nabla\psi}_{L^2(\Omega)}\leq\norm{\phi}_{L^2(\Omega)}.
    \end{equation*}
    It follows from standard elliptic theory (see, for example, \cite[Theorem 2.2.2.5]{grisvard}) that we even have $\psi\in H^2(\Omega)$. Moreover, it holds that
    \begin{equation*}
        \norm{\psi}_{H^2(\Omega)}\leq C\norm{\diverg\phi}_{L^2(\Omega)}\leq C\norm{\phi}_{H^1(\Omega)}.
    \end{equation*}
    Considering \eqref{weak-neumann} with $\xi=p(t)\in H^1(\Omega)$, and applying the momentum equation of \eqref{navier-stokes}, which holds in $L^2(\Omega)$ pointwise in time, we obtain
    \begin{align*}
        \abs{\int_{\Omega}\nabla p(t)\cdot\phi dx}&=\abs{\int_{\Omega}\nabla p(t)\cdot\nabla\psi dx}\\&\leq\abs{\int_{\Omega}\partial_t u(t)\cdot\nabla\psi dx}+\abs{\int_{\Omega} \diverg(u(t)\otimes u(t))\cdot\nabla\psi dx}+\nu\abs{\int_{\Omega}\Delta u(t)\cdot\nabla\psi dx}\\
        &=\abs{\int_{\Omega}\underbrace{\diverg(\partial_t  u(t))}_{=0}\psi dx}+\abs{\int_{\Omega} u(t)\otimes u(t):\nabla^2\psi dx}+\nu\abs{\int_{\Omega}\Delta u(t)\cdot\nabla\psi dx}\\
        &\leq \norm{u(t)}^2_{L^4(\Omega)}\norm{\psi}_{H^2(\Omega)}+\nu\norm{\Delta u(t)}_{L^2(\Omega)}\norm{\psi}_{H^1(\Omega)}\\
        &\leq C\norm{u(t)}^2_{L^4(\Omega)}\norm{\phi}_{H^1(\Omega)}+C\nu\norm{\Delta u(t)}_{L^2(\Omega)}\norm{\phi}_{L^2(\Omega)}.
    \end{align*}
    
    As we work in two dimensions, we can use the identity $\Delta u=\nabla^\perp\omega$, which concludes the proof.
\end{proof}
\subsection{Existence of strong solutions}
The following proposition is the heart of our article. It will give sufficient convergence for the statement of \cite[Lemma 3.3]{wiedemann} and yield additional regularity in the limit, which guarantees the existence of strong solutions for non-compatible initial data.
\begin{prop}\label{approximate-convergence}
    Fix $\nu>0$ and some $p\geq2$. Let $\{\omega_{0,n}\}_{n\in\N}\subset H^{1}(\Omega)\cap L^{\infty}(\Omega)$ with $u_{0,n}=K_{\Omega}(\omega_{0,n})\in W$ be a sequence of compatible functions such that
    \begin{equation*}
        \omega_{0,n}\to\omega_0 \text{ strongly in }L^{p}(\Omega) \text{ as }n\to\infty.
    \end{equation*}
    Then there exists for the corresponding sequence of solutions $u_{n}\in C([0,T];H^{2}(\Omega))$ to \eqref{navier-stokes} with $u_{n}(0,\cdot)=K_\Omega(\omega_{0,n})$ a subsequence $n_{k}\to\infty$ such that
    \begin{equation*}
        u_{n_k}\to u \text{ strongly in }C([0,T];L^2(\Omega))\cap L^2(0,T;H^1(\Omega)) \text{ as }k\to\infty,
    \end{equation*}
    where $u \in C([0,T];H)\cap L^2(0,T;V)$ is the unique solution to \eqref{navier-stokes} with $u(0,\cdot)=K_\Omega(\omega_{0})$. Moreover, we have the additional regularity 
    \begin{equation*}
        u \in L^2(0,T;H^2(\Omega))\text{ and }\partial_tu \in L^2(0,T;L^2(\Omega)).
    \end{equation*}
\end{prop}
\begin{proof}
    To begin, we define
    \begin{equation*}
        \overline{\omega}_n:=\omega_n-u_n\cdot\overline\tau\in C([0,T];H^1_0(\Omega))\cap L^2(0,T;H^2(\Omega)),
    \end{equation*}
    with $\overline{\tau} := (2\kappa-\alpha)\tau$, where $\alpha,\kappa$ and $\tau$ denote smooth extensions to $\overline{\Omega}$, and we used Lemma~\ref{additional-regularity} and~\eqref{vortexboundary}. The idea is to look at the so-called enstrophy balance for $\overline{\omega}_n$. Clearly, $\overline{\omega}_n$ is no longer a solution to the vorticity equation; however, the resulting error term is controllable. To calculate the error term, we use that $u_n$ solves \eqref{navier-stokes}:
     \begin{equation}\label{vorticity-error}
    \begin{aligned}
        \partial_t\overline{\omega}_n+u_n\cdot\nabla\overline{\omega}_n-\nu\Delta\overline{\omega}_n &= -\partial_t u_n\cdot\overline\tau-u_n\cdot\nabla\left(u_n\cdot\overline\tau\right)+\nu\Delta(u_n\cdot\overline\tau)\\
        &= \left((u_n\cdot\nabla)u_n+\nabla p_n-\nu\Delta u_n\right)\cdot\overline\tau-((u_n\cdot\nabla)u_n)\cdot\overline\tau-u_n\cdot(\nabla \overline\tau^Tu_n)\\
        &\;\;\;\;\;+\nu\Delta u_n\cdot\overline\tau+2\nu\text{trace}(\nabla u_n^T\nabla \overline\tau)+\nu u_n\cdot\Delta\overline\tau\\
        &=-u_n\cdot(\nabla \overline\tau^Tu_n)+\nabla p_n\cdot\overline\tau+2\nu\text{trace}(\nabla u_n^T\nabla \overline\tau)+\nu u_n\cdot\Delta\overline\tau\\
        &=:f(u_n,\nabla p_n,\overline\tau).
    \end{aligned}
\end{equation}

As usual, we multiply \eqref{vorticity-error} by $\overline{\omega_{n}}$ and integrate in space-time up to an arbitrary $t\in[0,T]$:
\begin{equation*}
    \int_{0}^{t}\int_{\Omega}\overline{\omega}_n\partial_{t}\overline{\omega}_n+\overline{\omega}_n(u_n\cdot\nabla\overline{\omega}_n)-\nu\overline{\omega}_n\Delta{\overline\omega_n} dx ds=\int_0^t\int_\Omega f(u_n,\nabla p_n,\overline\tau)\overline{\omega}_ndxds.
    \end{equation*}
Note that we have $\partial_{t}\overline{\omega}_n \in L^2(0,T;L^2(\Omega))$, and hence the chain rule $\frac{1}{2}\frac{d}{dt}\abs{\overline\omega_n}^2=\overline\omega_n\partial_t\overline\omega_n$ is applicable. Together with integration by parts and recalling that ${\overline\omega_n}$ is zero on the boundary, we obtain
    \begin{equation*}
        \int_{0}^{t}\int_{\Omega}\frac{1}{2}\frac{d}{ds}\abs{\overline\omega_n}^2+\nu\abs{\nabla\overline\omega_n}^2 dx ds=\int_{0}^{t}\int_{\Omega}f(u_n,\nabla p_n,\overline\tau)\overline{\omega}_n dxds.
    \end{equation*}
Next, we want to estimate the right hand-side term by term. With Ladyzhenskaya and Hölder,
\begin{align*}
    \int_{\Omega}-u_n\cdot(\nabla \overline\tau^Tu_n)\overline{\omega}_n dx&\leq \norm{\overline{\tau}}_{C^1(\Omega)}\norm{u_n}_{L^4(\Omega)}^2\norm{\overline{\omega}_n}_{L^2(\Omega)}\\&\leq C\norm{\overline{\tau}}_{C^1(\Omega)}\norm{u_n}_{L^2(\Omega)}\norm{\nabla u_n}_{L^2(\Omega)}\norm{\overline{\omega}_n}_{L^2(\Omega)}.
\end{align*}
For the pressure term, we apply Lemma \ref{pressure-lemma} to obtain
\begin{align*}
    \int_{\Omega}\nabla p_n\cdot\overline\tau\overline{\omega}_n dx&\leq C\norm{u_n}^2_{L^4(\Omega)}\norm{\overline{\tau}\overline{\omega}_n}_{H^1(\Omega)}+C\nu\norm{\nabla\omega_n}_{L^2(\Omega)}\norm{\overline{\tau}\overline{\omega}_n}_{L^2(\Omega)}\\&\leq C\norm{u_n}_{L^2(\Omega)}\norm{\nabla u_n}_{L^2(\Omega)}\norm{\overline{\tau}}_{C^1(\Omega)}\left(\norm{\overline{\omega}_n}_{L^2(\Omega)}+\norm{\nabla\overline{\omega}_n}_{L^2(\Omega)}\right)\\&\;\;\;\;+C\nu\norm{\nabla\omega_n}_{L^2(\Omega)}\norm{\overline{\tau}}_{L^\infty(\Omega)}\norm{\overline{\omega}_n}_{L^2(\Omega)}.
\end{align*}
Lastly,
\begin{align*}
    \int_{\Omega}\left(2\nu\text{trace}(\nabla u_n^T\nabla \overline\tau)+\nu u_n\cdot\Delta\overline\tau\right)\overline{\omega}_n dx\leq \nu C\norm{\overline{\tau}}_{C^2(\Omega)}\left(\norm{u_n}_{L^2(\Omega)}+\norm{\nabla u_n}_{L^2(\Omega)}\right)\norm{\overline{\omega}_n}_{L^2(\Omega)}.
\end{align*}
All in all, with the energy bound \eqref{energyestimate}, Lemma \ref{curl_estimate} and Young's inequality, we obtain that
\begin{align*}
    \int_{0}^{t}\int_{\Omega}f(u_n,\nabla p_n,\overline\tau)\overline{\omega}_n dxds\leq& \frac{\nu}{2}\int_{0}^{t}\norm{\nabla{\overline{\omega}_n}}_{L^2(\Omega)}^2ds\\&+C\left(\norm{\overline{\tau}}_{C^2(\Omega)},\norm{u_{0,n}}_{L^2(\Omega)},t,\nu\right)\int_0^t\left(1+\norm{\overline{\omega}_n}_{L^2(\Omega)}^2\right)ds.
\end{align*}
By absorbing the viscous term into the left hand-side, we can use Gronwall to get
    \begin{equation*}
        \int_\Omega\abs{\overline{\omega}_n(t)}^2dx+\nu\int_0^t\int_\Omega\abs{\nabla\overline{\omega}_n(t)}^2dxds\leq C\left(t,\nu,\norm{\overline\tau}_{C^2(\Omega)},\norm{\omega_{0,n}}_{L^{2}(\Omega)}\right).
    \end{equation*}
    The constant can be chosen independently of $n$, because the initial data converge.
    As the velocity is a Leray--Hopf solution, we obtain the same bound for the original vorticity
    \begin{equation}\label{lerayhopf-vorticity}
        \int_\Omega\abs{\omega_n(t)}^2dx+\nu\int_0^t\int_\Omega\abs{\nabla\omega_n}^2dxds\leq C\left(\nu,\norm{\omega_{0}}_{L^{2}(\Omega)}\right),
    \end{equation}
    where the right hand-side may grow to infinity as viscosity $\nu$ goes to 0. Lemma \ref{riesz-estimate} allows us to upgrade \eqref{lerayhopf-vorticity} to a uniform-in-$n$ bound for $\norm{u_n}_{L^2(0,T;H^2(\Omega))}$. Indeed, since $\diverg u_n=0$, in two dimensions $\Delta u_n =\nabla^\perp\curl u_n= \nabla^\perp\omega_n$, which yields
    \begin{equation}\label{L2H2-bound}
        \begin{aligned}
            \norm{u_n}_{L^2(0,T;H^2(\Omega))}&\leq C\left(\norm{\Delta u_n}_{L^2(0,T;L^2(\Omega))}+\norm{u_n}_{L^2(0,T;L^2(\Omega))}\right)\\
            &=C\left(\norm{\nabla^\perp \omega_n}_{L^2(0,T;L^2(\Omega))}+\norm{u_n}_{L^2(0,T;L^2(\Omega))}\right)\\
            &\leq C\left(\nu,\norm{\omega_{0}}_{L^{2}(\Omega)}\right).
        \end{aligned}
    \end{equation}
    Next, we look at the time derivative, where we use the Leray projection of the momentum equation in \eqref{navier-stokes} to get
    \begin{equation}\label{time-l2}
        \begin{aligned}
            \norm{\partial_t u_n}_{L^2(0,T;L^2(\Omega))}=&\norm{ -\mathbb{P}[(u_n\cdot\nabla)u_n]+\nu\mathbb{P}\Delta u_n}_{L^2(0,T;L^2(\Omega))}\\
            \leq &\norm{(u_n\cdot\nabla)u_n}_{L^2(0,T;L^2(\Omega))}+\norm{\nu\Delta u_n}_{L^2(0,T;L^2(\Omega))}\\
            \leq & \norm{u_n}_{L^\infty(0,T;L^4(\Omega)}\norm{\nabla u_n}_{L^2(0,T;L^4(\Omega))}+\nu\norm{\nabla \omega_n}_{L^2(0,T;L^2(\Omega))}\\
            \leq&C\norm{u_n}_{L^\infty(0,T;H^1(\Omega)}\norm{u_n}_{L^2(0,T;H^2(\Omega))}+\nu\norm{\nabla \omega_n}_{L^2(0,T;L^2(\Omega))}\\
            \leq& C\left(\nu,\norm{\omega_{0}}_{L^{2}(\Omega)}\right).
        \end{aligned}
    \end{equation}
    
    From the Aubin--Lions lemma \ref{aubin-lions} with $X_0=H^1(\Omega)$, $X=L^2(\Omega)$, $X_1=L^{2}(\Omega)$ and $X_0=H^2(\Omega)$, $X=H^1(\Omega)$, $X_1=L^{2}(\Omega)$, we obtain a subsequence $n_k\to\infty$ such that
    \begin{equation*}
        u_{n_k}\to u \text{ strongly in }C([0,T];L^2(\Omega))\cap L^2(0,T;H^1(\Omega)) \text{ as }k\to\infty,
    \end{equation*}
    for some $u\in C([0,T];H)\cap L^2(0,T;H^1(\Omega))$. With the same argument as in \cite[Proposition 1]{nussenzveig}, the limit is in fact a weak solution to \eqref{navier-stokes}. Without loss of generality we can assume with \eqref{L2H2-bound} that the sequence also converges weakly in $L^2((0,T);H^2(\Omega))$ and thus we have $u\in L^2(0,T;H^2(\Omega))$. Finally, we show the additional time regularity. With the Banach--Alaoglu theorem and the bound \eqref{time-l2}, there exists a $g\in L^2(0,T;L^2(\Omega))$ and a subsequence, which we again label with $n_k$, such that
    \begin{equation*}
        \partial_t u_{n_k} \rightharpoonup g \text{ in }L^2(0,T;L^2(\Omega)) \text{ as }k\to\infty.
    \end{equation*}
    We want to show $g=\partial_tu$ distributionally. Fix $\psi\in C^\infty_c((0,T))$ and $\phi\in L^2(\Omega)$. For each $n_k$, we know that $u_{n_k}\in H^1((0,T);L^2(\Omega))$, and thus we have
    \begin{equation}\label{weak-time-derivative}
        \int_0^T\int_\Omega u_{n_k}(t,x)\cdot\phi(x)dx\psi'(t)dt= - \int_0^T\int_\Omega \partial_tu_{n_k}(t,x)\cdot\phi(x)dx\psi(t)dt.
    \end{equation}
    With our already established convergence, we can pass to the limit in \eqref{weak-time-derivative} 
    \begin{equation*}
        \int_0^T\int_\Omega u(t,x)\cdot\phi(x)dx\psi'(t)dt= - \int_0^T\int_\Omega g(t,x)\cdot\phi(x)dx\psi(t)dt,
    \end{equation*}
    which concludes the proof.
\end{proof}
\begin{proof}[Proof of Theorem \ref{theorem-strong-solutions}]
    Essentially, we have already shown everything in Proposition \ref{approximate-convergence}. We apply \cite[Section 5.9.2]{evans} with  $u \in L^2(0,T;H^2(\Omega))$ and $\partial_tu \in L^2(0,T;L^2(\Omega))$ to see that $u\in C([0,T];H^1(\Omega))$. It remains to show that $u$ satisfies the Navier boundary condition. Let us write the Navier boundary condition as a first-order boundary operator
    \begin{equation*}
        B:H^2(\Omega)\to H^{1/2}(\partial\Omega),\quad u\mapsto 2(Du)_{S}n\cdot\tau + \alpha u\cdot \tau,
    \end{equation*}
    which is bounded and linear.
    Using the same sequence as in Proposition \ref{approximate-convergence}, we have that 
    \begin{equation*}
        u_n \rightharpoonup u \text{ in }L^2(0,T;H^2(\Omega)),
    \end{equation*}
    and hence also
    \begin{equation*}
        B(u_n) \rightharpoonup B(u) \text{ in }L^2(0,T;H^{1/2}(\partial\Omega)).
    \end{equation*}
    For each $n\in\N$, we know that $B(u_n)=0$, which implies, with uniqueness of weak limits, that $B(u)=0$. Thus, the limit $u$ satisfies the Navier boundary condition. 
\end{proof}
\section{Vanishing viscosity limit for the vorticity}
Finally, we can show the analogue of \cite[Lemma 3.3]{wiedemann}, which is precisely the lemma stated below, to use \cite[Theorem 1.3]{wiedemann}.
\begin{lem}\label{wiedemann-lemma}
    Fix $\nu>0$ and let $\omega^{\nu}= \curl u^\nu$ be the associated vorticity to the unique solution $u^\nu$ of \eqref{navier-stokes} with $u^{\nu}(0,\cdot)=K_\Omega(\omega_{0}^{\nu})$ for $\omega_{0}^{\nu}\in L^p(\Omega)$ with $p>2$ and $q\in[1,p)$. Then, for any nonnegative function $\phi\in C^{\infty}_{c}([0,T)\times\Omega)$, it holds that 
    \begin{equation}\label{3.6}
        0\leq \int_{0}^{T}\int_{\Omega}\abs{\omega^{\nu}}^q(\partial_{t}\phi+u^{\nu}\cdot\nabla\phi)dxdt+\int_{\Omega}\abs{\omega_{0}^{\nu}}^q\phi(0,\cdot)dx+\nu C\left(\norm{\omega_{0}^{\nu}}_{L^{p}(\Omega)}\right),
    \end{equation}
    for some constant $C>0$.
\end{lem}
\begin{proof}
    We approximate $\omega^{\nu}$ as in Proposition \ref{approximate-convergence} by a sequence $\omega_n^\nu$ of solutions for compatible initial data. Next, we multiply \eqref{navier-stokes-vorticity} with $\abs{\omega^{\nu}_n}^{q-2}\omega^{\nu}_n\phi$ and integrate in space:
    \begin{equation}\label{sgn-formulation}
        \int_{\Omega} \partial_{t}\omega^{\nu}_n\abs{\omega^{\nu}_n}^{q-2}\omega^{\nu}_n\phi + u^{\nu}_n\cdot\nabla \omega^{\nu}_n\abs{\omega^{\nu}_n}^{q-2}\omega^{\nu}_n\phi-\nu\Delta \omega^{\nu}_n\abs{\omega^{\nu}_n}^{q-2}\omega^{\nu}_n\phi dx=0.
    \end{equation}
    With integration by parts, where we need $\omega^{\nu}_n\in L^2(0,T;H^2(\Omega))$ (Lemma \ref{additional-regularity}), we get 
    \begin{equation*}
        \frac{d}{dt}\int_{\Omega}\abs{\omega^{\nu}_n}^q\phi dx + \nu q(q-1)\int_\Omega\abs{\omega_n^\nu}^{q-2}\abs{\nabla\omega_n^\nu}^2 \phi dx = \int_{\Omega}\abs{\omega^{\nu}_n}^q\left(\partial_{t}\phi+u^{\nu}_n\cdot\nabla\phi+\nu\Delta\phi\right)dx.
    \end{equation*}
    After integrating in time and using the non-negativity of the first term on the left hand side, it follows that
    \begin{equation*}
        0\leq \int_{0}^{T}\int_{\Omega}\abs{\omega^{\nu}_n}^q\left(\partial_{t}\phi+u^{\nu}_n\cdot\nabla\phi\right)dxdt+\int_{\Omega}\abs{\omega_{0,n}^{\nu}}^q\phi(0,\cdot)dxdt+\nu\int_{0}^{T}\int_{\Omega}\abs{\omega^{\nu}_n}^q\Delta\phi dxdt.
    \end{equation*}
    In order to estimate the last term, we use the uniform vorticity bound \eqref{uniform-vorticity-bound}
    \begin{equation*}
        \abs{\nu\int_{0}^{T}\int_{\Omega}\abs{\omega^{\nu}_n}^q\Delta\phi dx}\lesssim\nu\norm{\omega^{\nu}_n}_{L^{\infty}(0,T;L^{q}(\Omega))}^q\leq \nu C\left(\norm{\omega^{\nu}_{0,n}}_{L^{q}(\Omega)}\right).
    \end{equation*}
    So we have shown \eqref{3.6} for the approximate solution $\omega^\nu_n$. Let us take the limit to see that it also holds for $\omega^\nu$. From \eqref{uniform-vorticity-bound} we know that $\abs{\omega^\nu_n}^q$ is uniformly bounded in $L^{p/q}(\Omega)$ with $\frac{p}{q}>1$ and thus it is also uniformly integrable. Together with the strong convergence in Proposition \ref{approximate-convergence}, which implies convergence almost everywhere after choosing another subsequence if necessary, we invoke Vitali's convergence theorem to get
    \begin{equation}\label{norm-convergence}
        \abs{\omega^{\nu}_n}^q\to\abs{\omega^{\nu}}^q \text{ in }L^1((0,T)\times\Omega) \text{ as }n\to\infty
    \end{equation}
    and therefore
    \begin{equation*}
        \int_{0}^{T}\int_{\Omega}\abs{\omega^{\nu}_n}^q\partial_{t}\phi dxdt\to \int_{0}^{T}\int_{\Omega}\abs{\omega^{\nu}}^q\partial_{t}\phi dxdt \text{ as }n\to\infty.
    \end{equation*}
   One argues similarly for the viscous term. The convergence of the initial term follows directly from the $L^p$ convergence of the initial data. Lastly, we need to check the convective term. We start by splitting it into the three terms
    \begin{equation*}
        \begin{aligned}
            \int_{0}^{T}\int_{\Omega}\abs{\omega^{\nu}_n}^qu^{\nu}_n\cdot\nabla\phi dxdt=&\int_{0}^{T}\int_{\Omega}\left(\abs{\omega^{\nu}_n}^q-\abs{\omega^{\nu}}^q\right)u^{\nu}_n\cdot\nabla\phi dxdt\\&+\int_{0}^{T}\int_{\Omega}\abs{\omega^{\nu}}^q\left(u^{\nu}_n - u^{\nu}\right)\cdot\nabla\phi dxdt
            +\int_{0}^{T}\int_{\Omega}\abs{\omega^{\nu}}^qu^{\nu}\cdot\nabla\phi dxdt.
        \end{aligned}
    \end{equation*}
    The first term vanishes due to the moreover part of Corollary \ref{uniform-velocity-corollary}:
    \begin{equation*}
        \begin{aligned}
            &\int_{0}^{T}\int_{\Omega}\left(\abs{\omega^{\nu}_n}^q-\abs{\omega^{\nu}}^q\right)u^{\nu}_n\cdot\nabla\phi dxdt\\&\leq \norm{\abs{\omega^{\nu}_n}^q-\abs{\omega^{\nu}}^q}_{L^1((0,T)\times\Omega)}\norm{u^{\nu}_n}_{L^\infty((0,T)\times\Omega)}\norm{\nabla\phi}_{L^\infty((0,T)\times\Omega)}\\
            &\to 0 \text{ as }n\to\infty.
        \end{aligned}
    \end{equation*}
    We have seen in the proof of Proposition \ref{approximate-convergence} that $u_n^\nu$ converges in $L^2(0,T;H^1(\Omega))$ and therefore also in $L^2(0,T;L^r(\Omega))$ for every finite $r<\infty$. Let $s'=(\frac{p}{q})'<\infty$ be the dual exponent of $s=\frac{p}{q}$, then we get for the second term
    \begin{equation*}
        \begin{aligned}
            \int_{0}^{T}\int_{\Omega}\abs{\omega^{\nu}}^q\left(u^{\nu}_n - u^{\nu}\right)\cdot\nabla\phi dxdt\leq& \int_{0}^{T}\norm{\abs{\omega^{\nu}}^q}_{L^s(\Omega)}\norm{u^{\nu}_n - u^{\nu}}_{L^{s'}(\Omega)}\norm{\nabla\phi}_{L^\infty(\Omega)}dt\\
            \leq& \norm{\omega^{\nu}}_{L^\infty(0,T;L^p(\Omega))}^q\norm{u^{\nu}_n - u^{\nu}}_{L^1(0,T;L^{s'}(\Omega))}\norm{\nabla\phi}_{L^\infty((0,T)\times\Omega)}\\
            &\to 0 \text{ as }n\to\infty,
        \end{aligned}
    \end{equation*}
    which concludes the proof.
\end{proof}
Before concluding with the proof of Proposition \ref{strong-convergence-vorticity}, we need the following renormalization property, which will help us later to rule out possible concentrations at the endpoint $q=p$.
\begin{lem}\label{convex-renormalization-property}
    Let $\nu\in(0,1)$ and $\omega^{\nu}= \curl u^\nu$ be the associated vorticity to the unique solution $u^\nu$ of \eqref{navier-stokes} with $u^{\nu}(0,\cdot)=K_\Omega(\omega_{0}^{\nu})$ for $\omega_{0}^{\nu}\in L^p(\Omega)$ with $p>2$. There exists a $R_0>0$, independent of viscosity $\nu$, such that the following holds. Let $\beta\in C^2(\R)$ satisfy
    \begin{equation}\label{beta-conditions}
        \beta \geq0,\;\beta''\geq0,\; \beta(x)=0 \text{ for }x\in[-R_0,R_0],\text{ and } \abs{\beta'(x)}\leq 1+\abs{x}^{p-1} \text{ for all }x\in\R.
    \end{equation}
    Then we have 
    \begin{equation*}
        \int_\Omega\beta(\omega^\nu(t))dx\leq\int_\Omega\beta(\omega^\nu_0)dx \text{ for all }t\in[0,T].
    \end{equation*}
\end{lem}
\begin{proof}
    Pick a sequence of compatible functions such that 
    \begin{equation*}
        \omega_{0,n}^\nu\to\omega_{0}^\nu \text{ in }L^p(\Omega).
    \end{equation*}
    For the corresponding solutions $\omega^\nu_n$, there exists, with Corollary \ref{uniform-velocity-corollary} and the convergence of initial data, a $R_0> 0$, independent of $n$ and $\nu$, such that
    \begin{equation*}
        \abs{\omega^\nu_n}=\abs{(2\kappa-\alpha)u_n^\nu\cdot\tau}\leq R_0 \text{ on }(0,T)\times\partial\Omega.
    \end{equation*}
    Let $\beta\in C^2(\R)$ satisfy \eqref{beta-conditions} for $R_0$. By the choice of $R_0$, $\beta'(\omega_n^\nu)$ vanishes on the boundary. The regularity in Theorem \ref{theorem-clopeau} justifies to test the vorticity equation with $\beta'(\omega_n^\nu)$ and integrate by parts to obtain
    \begin{equation*}
        \int_{\Omega}\beta(\omega_n^\nu(t))dx+\nu\int_0^t\int_{\Omega}\beta''(\omega_{n}^\nu)\abs{\nabla\omega_n^\nu}^2dxds=\int_{\Omega}\beta(\omega_{0,n}^\nu) dx,
    \end{equation*}
    which implies
     \begin{equation*}
        \int_\Omega\beta(\omega_n^\nu(t))dx\leq\int_\Omega\beta(\omega^\nu_{0,n})dx \text{ for all }t\in[0,T].
    \end{equation*}
    Next, we want to pass to the limit $n\to\infty$. For the right hand-side, we use the fundamental theorem of calculus, Hölder's inequality, and the growth condition in \eqref{beta-conditions} to get
    \begin{align*}
        \int_{\Omega}\abs{\beta(\omega^\nu_{0,n})-\beta(\omega^\nu_0)}dx &= \int_{\Omega}\abs{(\omega^\nu_{0,n}-\omega^\nu_0)\int_0^1\beta'(\omega^\nu_0+\theta[\omega^\nu_{0,n}-\omega^\nu_0])d\theta}dx\\
        &\leq C\int_{\Omega}\abs{(\omega^\nu_{0,n}-\omega^\nu_0)}\left(1+\abs{\omega^\nu_0}^{p-1}+\abs{\omega^\nu_{0,n}}^{p-1}\right)dx
        \\&\leq C\norm{\omega^\nu_{0,n}-\omega^\nu_0}_{L^p(\Omega)}\left(\abs{\Omega}^{1-1/p}+\norm{\omega^\nu_0}_{L^p(\Omega)}^{p-1}+\norm{\omega^\nu_{0,n}}_{L^p(\Omega)}^{p-1}\right)\\
        &\to 0 \text{ for }n\to\infty.
    \end{align*}
    For the left hand-side, we need to be a bit cautious, as we have not yet established the continuity of the limit map $t\mapsto\int_{\Omega}\beta(\omega^\nu(t))dx$. From the convergence in Proposition \ref{approximate-convergence} for a subsequence, which we label again with $n$, we know that 
    \begin{equation*}
        \omega^\nu_n(t,x)\to \omega^\nu(t,x)\text{ for almost every }(t,x)\in(0,T)\times \Omega.
    \end{equation*}
    Hence, there exists by Fubini's theorem a measurable set $G_\nu\subset (0,T)$ with $\abs{G_{\nu}}=T$, such that for every $t\in G_{\nu}$ it holds that $\omega^\nu_n(t,x)\to \omega^\nu(t,x)\text{ for almost every }x\in \Omega$. With Fatou's lemma, we get for $t\in G_\nu$ that
    \begin{align*}
        \int_{\Omega}\beta(\omega^\nu(t))dx\leq \liminf_{n\to\infty}\int_{\Omega}\beta(\omega^\nu_n(t))dx\leq\liminf_{n\to\infty}\int_{\Omega}\beta(\omega^\nu_{0,n})dx=\int_{\Omega}\beta(\omega^\nu_{0})dx.
    \end{align*}
    We now extend it to every time point. Let $t\in[0,T]$ be arbitrary. Then there exists a sequence $t_j\in G_\nu$ with $t_j\to t$. By Theorem \ref{theorem-strong-solutions}, we have $\omega^\nu\in C([0,T];L^2(\Omega))$, which implies for a subsequence, again labeled with $j$, that
    \begin{equation*}
        \omega^\nu(t_j,x)\to\omega^\nu(t,x) \text{ for almost every }x\in\Omega.
    \end{equation*}
    Therefore, we can again use Fatou's lemma to conclude the proof
    \begin{equation*}
        \int_{\Omega}\beta(\omega^\nu(t))dx\leq \liminf_{j\to\infty}\int_{\Omega}\beta(\omega^\nu(t_j))dx\leq\int_{\Omega}\beta(\omega^\nu_{0})dx.
    \end{equation*}
\end{proof}
\begin{proof}[Proof of Proposition \ref{strong-convergence-vorticity}]
    The proof proceeds in two steps. First, we use \cite[Theorem 1.3]{wiedemann} to show convergence in $C([0,T];L^q(\Omega))$ for $q<p$. Second, we upgrade to the endpoint $q=p$ by controlling the $p$-tails with Lemma \ref{convex-renormalization-property}.\\
    Given $T\in(0,\infty)$, we pick a slightly larger time $\tilde T\in(T,\infty)$. In contrast to the assumption \eqref{local-p-enstrophies} in \cite{wiedemann}, the uniform vorticity bound \eqref{global-p-enstrophies} is not specified to $T$. Therefore, we can apply\footnote{In fact, this theorem assumes the Dirichlet boundary condition. However, the latter only serves to provide a uniform-in-$\nu$ bound for the velocities in $L^\infty(0,T;L^2(\Omega))$, which remains true under the Navier condition by virtue of Theorem~\ref{energybound} above.} \cite[Theorem 1.3]{wiedemann} for $\tilde T$ to get a sequence $\nu^k\to 0$ such that 
    \begin{equation*}
        \omega^{\nu_{k}}\to \omega \text{ strongly in } C_{loc}([0,\tilde T);L^{q}_{loc}(\Omega))\text{ as }k\to\infty \text{ for any }q\in[1,p),\text{ }
    \end{equation*}
    and $\omega$ is a weak solution of \eqref{euler-vorticity} for $\omega(0,\cdot)=\omega_{0}$. Especially, we have that 
    \begin{equation*}
        \omega^{\nu_{k}}\to \omega \text{ strongly in } C([0,T];L^{q}_{loc}(\Omega))\text{ as }k\to\infty \text{ for any }q\in[1,p).
    \end{equation*}
    Let $q\in[1,p)$ be arbitrary. Suppose the convergence in $C([0,T];L^{q}(\Omega))$ failed. Then there must exist a time $T'\in[0,T]$ and a sequence $\set{B_k}\subset \Omega$ with $\abs{B_k}\to 0$ as $k\to\infty$ such that the $L^q$-difference on $B_k$ does not vanish as $k\to\infty$, i.e., there exists a $\delta>0$ such that
           \begin{equation*}
                \sup_{t\in[0,T']}\int_{B_k}\abs{\omega^{\nu_k}(t)-\omega(t)}^qdx>\delta \text{ for all }k\in\N.
           \end{equation*}
           However, this contradicts the uniform \(L^p\)-bound for the vorticity \eqref{uniform-vorticity-bound} (since \(q<p\)):
           \begin{equation*}
               \begin{aligned}
                   \int_{B_k}\abs{\omega^{\nu_k}(t)-\omega(t)}^qdx=&\int_\Omega \mathds{1}_{B_k}\abs{\omega^{\nu_k}(t)-\omega(t)}^qdx\\
                   \leq& \abs{B_k}^{1-\frac{q}{p}}\norm{\omega^{\nu_k}(t)-\omega(t)}^q_{L^p(\Omega)}\\\leq&\abs{B_k}^{1-\frac{q}{p}}\left(\norm{\omega^{\nu_k}(t)}_{L^p(\Omega)}+\norm{\omega(t)}_{L^p(\Omega)}\right)^q\\
                   \leq&\abs{B_k}^{1-\frac{q}{p}}C\left(\norm{\omega^{\nu_k}_0}_{L^p(\Omega)},\norm{\omega_0}_{L^p(\Omega)}\right)\to0 \text{ as }k\to\infty,
               \end{aligned}
           \end{equation*}
               because the initial data converge and are therefore uniformly bounded in $\nu_k$. Thus, we have 
        \begin{equation*}
        \omega^{\nu_{k}}\to \omega \text{ strongly in } C([0,T];L^{q}(\Omega))\text{ as }k\to\infty \text{ for any }q\in[1,p).
    \end{equation*}
    For the endpoint $q=p$, we introduce the truncation map
    \begin{equation*}
        T_Rx:=\begin{cases}
            -R,& x<-R,\\
            x,&\abs{x}\leq R,\\
            R,&x>R,
        \end{cases}
    \end{equation*}
    with the immediate properties
    \begin{equation}\label{truncation-properties}
        \abs{T_Rx-T_Ry}\leq \abs{x-y}\text{ and }\abs{T_Rx-T_Ry}\leq 2R.
    \end{equation}
    Observe that $\abs{x-T_Rx}=(\abs{x}-R)_+$, where $(f)_+=\max\set{f,0}$ is the positive part. The idea is to use Lemma \ref{convex-renormalization-property} with $\beta_{R}(x)=\frac{1}{p}(\abs{x}-R)_+^p\in C^2(\R)$ for sufficiently large $R>0$. Indeed, one can easily verify that there is a $R_0>0$ such that $\beta_R$ satisfies \eqref{beta-conditions} for all $R\geq R_0$ and thus we have
    \begin{equation}\label{vorticity-tail}
        \int_\Omega(\abs{\omega^\nu(t)}-R)_+^pdx\leq\int_\Omega(\abs{\omega^\nu_0}-R)_+^pdx \text{ for all }R\geq R_0 \text{ and }t\in[0,T].
    \end{equation}
    We have now collected all the tools to show convergence in $C([0,T];L^p(\Omega))$. Let $\eps>0$ be arbitrary and $\nu_k\to0$ the same subsequence as above. We pick a $R>0$ sufficiently large such that
    \begin{equation*}
        \sup_k\int_\Omega(\abs{\omega^{\nu_k}_0}-R)_+^pdx<\left(\frac{\eps}{3}\right)^{p},
    \end{equation*}
    which is possible as the initial vorticities converge strongly in $L^p$ and thus are uniformly bounded. By the convergence in $C([0,T];L^2(\Omega))$, there exists a $k_0\in \N$ such that for all $k\geq k_0$ it holds that
    \begin{equation*}
        \norm{\omega^{\nu_k}-\omega}_{C([0,T];L^2(\Omega))}<\left(\frac{\eps}{3(2R)^{1-2/p}}\right)^{p/2}.
    \end{equation*}
    Together, we obtain with \eqref{truncation-properties} and \eqref{vorticity-tail} that
    \begin{align*}
        &\norm{\omega^{\nu_k}-\omega}_{C([0,T];L^p(\Omega))}\\&\leq \norm{\omega^{\nu_k}-T_R\omega^{\nu_k}}_{C([0,T];L^p(\Omega))}+\norm{T_R\omega^{\nu_k}-T_R\omega}_{C([0,T];L^p(\Omega))}+\norm{T_R\omega-\omega}_{C([0,T];L^p(\Omega))}\\
        &=\norm{(\abs{\omega^{\nu_k}}-R)_+}_{C([0,T];L^p(\Omega))}+\norm{T_R\omega^{\nu_k}-T_R\omega}_{C([0,T];L^p(\Omega))}+\norm{(\abs{\omega}-R)_+}_{C([0,T];L^p(\Omega))}\\
        &\leq\norm{(\abs{\omega^{\nu_k}_0}-R)_+}_{L^p(\Omega)}+(2R)^{(p-2)/p}\norm{\omega^{\nu_k}-\omega}_{C([0,T];L^2(\Omega))}^{2/p}+\norm{(\abs{\omega_0}-R)_+}_{L^p(\Omega)}\\
        &< \frac{\eps}{3}+\frac{\eps}{3}+\frac{\eps}{3}=\eps,
    \end{align*}
    which shows the endpoint $q=p$. It remains to show \begin{equation*}
        \lim_{k\to\infty}\nu_k\int_0^T\int_{\Omega}\abs{\nabla\omega^{\nu_k}}^2dxdt=0.
    \end{equation*}
    As in Proposition \ref{approximate-convergence}, we consider the shifted vorticity $\overline{\omega^{\nu_k}}=\omega^{\nu_k}-u^{\nu_k}\cdot\overline{\tau}$, which fulfills
    \begin{equation}\label{shifted-enstrophy}
        \frac{1}{2}\norm{\overline{\omega^{\nu_k}}(T)}^2_{L^2(\Omega)}+\nu_k\int_0^T\norm{\nabla\overline{\omega^{\nu_k}}}^2_{L^2(\Omega)}ds=\frac{1}{2}\norm{\overline{\omega^{\nu_k}_0}}^2_{L^2(\Omega)}+\int_0^T\int_{\Omega}f^{\nu_k}\overline{\omega^{\nu_k}}dxds,
    \end{equation}
    with
    \begin{equation*}
        f^{\nu_k}=-u^{\nu_k}\cdot(\nabla \overline\tau^Tu^{\nu_k})+\nabla p^{\nu_k}\cdot\overline\tau+2\nu_k\text{trace}\left([\nabla u^{\nu_k}]^T\nabla \overline\tau\right)+\nu_k u^{\nu_k}\cdot\Delta\overline\tau.
    \end{equation*}
    The testing for \eqref{shifted-enstrophy} is justified, because the regularity obtained in Theorem \ref{theorem-strong-solutions} gives
    \begin{equation*}
        \overline{\omega^{\nu_k}}\in C([0,T];L^2(\Omega))\cap L^2(0,T;H^1_0(\Omega)),\;\partial_t\overline{\omega^{\nu_k}}\in L^2(0,T;H^{-1}(\Omega)).
    \end{equation*}
    We want to take the limit in \eqref{shifted-enstrophy}. From the already established convergence and Lemma \ref{curl_estimate}, we know
    \begin{equation*}
        u^{\nu_k}\to u \text{ in }C([0,T];W^{1,p}(\Omega))\text{ and } \overline{\omega^{\nu_k}}\to \overline{\omega}:=\omega-u\cdot\overline{\tau} \text{ in }C([0,T];L^2(\Omega))\text{ for }k\to\infty,
    \end{equation*}
    so it remains to show strong convergence for the pressure gradient. Before doing so, we observe that
    \begin{equation*}
        \norm{(I_2-\mathbb{P})\Delta u^{\nu_k}(t)}_{L^2(\Omega)}\leq C\norm{u^{\nu_k}(t)}_{H^1(\Omega)} \text{ for almost every }t\in[0,T],
    \end{equation*}
    where $C>0$ only depends on the domain $\Omega$ and the friction coefficient $\alpha$. Indeed, $\overline{\omega^{\nu_k}}(t)\in H^1_0(\Omega)$ implies $\nabla^\perp\overline{\omega^{\nu_k}}(t)\in H$ (see \cite[Theorem 1.2]{temam}) and thus
    \begin{align*}
        \norm{(I_2-\mathbb{P})\Delta u^{\nu_k}(t)}_{L^2(\Omega)}&=\norm{(I_2-\mathbb{P})\nabla^{\perp} \omega^{\nu_k}(t)}_{L^2(\Omega)}\\&=\norm{(I_2-\mathbb{P})\nabla^{\perp} \overline{\omega^{\nu_k}}(t)+(I_2-\mathbb{P})\nabla^{\perp}(u^{\nu_k}(t)\cdot\overline{\tau})}_{L^2(\Omega)}\\&=\norm{(I_2-\mathbb{P})\nabla^{\perp}(u^{\nu_k}(t)\cdot\overline{\tau})}_{L^2(\Omega)}\leq C\norm{u^{\nu_k}(t)}_{H^1(\Omega)}.
    \end{align*}
    Applying $(I_2-\mathbb{P})$ to the momentum equation in \eqref{navier-stokes}, we obtain
    \begin{align*}
        \nabla p^{\nu_k}=-(I_2-\mathbb{P})[(u^{\nu_k}\cdot\nabla)u^{\nu_k}]+\nu_k(I_2-\mathbb{P})\Delta u^{\nu_k}\to -(I_2-\mathbb{P})[(u\cdot\nabla)u]=\nabla p\text{ in }L^2((0,T)\times\Omega).
    \end{align*}
    Consequently,
    \begin{equation*}
        f^{\nu_k}\to f:= -u\cdot(\nabla \overline\tau^Tu)+\nabla p\cdot\overline\tau \text{ in }L^2((0,T)\times\Omega).
    \end{equation*}
    In summary, we have shown
    \begin{equation*}
        \lim_{k\to\infty}\nu_k\int_0^T\norm{\nabla\overline{\omega^{\nu_k}}}^2_{L^2(\Omega)}ds=\frac{1}{2}\norm{\overline{\omega_0}}^2_{L^2(\Omega)}+\int_0^T\int_{\Omega}f\overline{\omega}dxds-\frac{1}{2}\norm{\overline{\omega}(T)}^2_{L^2(\Omega)}.
    \end{equation*}
    The limiting shifted vorticity $\overline{\omega}$ solves
    \begin{equation*}
        \partial_t\overline{\omega}+u\cdot\nabla\overline{\omega}=f,
    \end{equation*}
    with $u\in L^{\infty}(0,T;V)$ and $f\in L^2((0,T)\times\Omega)$. Therefore, the standard renormalization argument (see, for example, \cite[Theorem 3.1]{boyer}), together with a cutoff argument (like in the proof of \cite[Theorem II.2]{diperna-lions}), yields
    \begin{equation*}
        \frac{1}{2}\norm{\overline{\omega}(T)}^2_{L^2(\Omega)}=\frac{1}{2}\norm{\overline{\omega_0}}^2_{L^2(\Omega)}+\int_0^T\int_{\Omega}f\overline{\omega}dxds,
    \end{equation*}
    which entails
    \begin{equation*}
        \lim_{k\to\infty}\nu_k\int_0^T\norm{\nabla\overline{\omega^{\nu_k}}}^2_{L^2(\Omega)}ds=0.
    \end{equation*}
    Finally, we close the proof by using again Corollary \ref{uniform-velocity-corollary}
    \begin{equation*}
        \nu_k\int_0^T\norm{\nabla\omega^{\nu_k}}^2_{L^2(\Omega)}ds\leq 2\nu_k\int_0^T\norm{\nabla\overline{\omega^{\nu_k}}}^2_{L^2(\Omega)}ds+2\nu_k\int_0^T\norm{\nabla(u^{\nu_k}\cdot\overline{\tau})}^2_{L^2(\Omega)}ds\to 0 \text{ as }k\to\infty.
    \end{equation*}
\end{proof}
\appendix
\section{Results from Functional Analysis}
\begin{lem}\label{curl_estimate}
    Let $\Omega\subset\R^{2}$ be a bounded, simply connected domain with a smooth boundary and $p\in(1,\infty)$. Then, for any $u\in V$ with $\curl u\in L^{p}(\Omega)$, it holds that
    \begin{equation*}
        \norm{\nabla u}_{L^{p}(\Omega)}\leq C(\Omega,p)\norm{\curl u}_{L^{p}(\Omega)}.
    \end{equation*}
\end{lem}
\begin{proof}
    As $u$ is divergence-free, there exists a stream function $\psi\in H^2(\Omega)\cap H^1_0(\Omega)$ such that $u=\nabla^\perp\psi$. For the construction, see, for example, \cite[(3.1)]{kelliher}. With a standard Calder\'on--Zygmund estimate, we get
    \begin{equation*}
        \norm{\nabla u}_{L^{p}(\Omega)}\leq \norm{\psi}_{W^{2,p}(\Omega)}\leq C(\Omega,p)\norm{\Delta\psi}_{L^{p}(\Omega)}=C(\Omega,p)\norm{\curl u}_{L^{p}(\Omega)}.
    \end{equation*}
\end{proof}
\begin{prop}[Aubin--Lions lemma \cite{aubin-lions-cite}]\label{aubin-lions}
    Let $X_{0}\subset X\subset X_{1}$ be Banach spaces such that the embedding $X_{0}\subset X$ is compact and the embedding $X\subset X_1$ is continuous. Then
    \begin{equation*}
        \set{u\in L^{p}(0,T;X_{0}), \partial_{t}u\in L^{q}(0,T;X_{1})}
    \end{equation*}
    embeds compactly into
    \begin{itemize}
        \item $L^{p}(0,T;X)$, for $p<\infty$ and $q\in[1,\infty]$,
        \item $C([0,T];X)$, for $p=\infty$ and $q\in(1,\infty]$.
    \end{itemize}
\end{prop}
\section{Agmon--Douglis--Nirenberg theory}
In this section, we recall some basic notions of the relevant Agmon--Douglis--Nirenberg (ADN) theory \cite{agmon} tailored to our context, primarily following \cite[Appendix D]{bochev}. Let $\Omega\subset\R^d$ be a bounded domain. We consider an elliptic partial differential equation of the form 
\begin{equation}\label{elliptic-system}
    \left\{
	\begin{aligned}
		&L(x,D)u=f \text{ in }\Omega,\\
        &B(x,D)u=g \text{ on }\partial\Omega,
	\end{aligned}\right.
\end{equation}
for an unknown $u:\Omega\to\R^M$. The operators $L(x,D)$ and $B(x,D)$ take values in $M\times M$, $L\times M$ and their entries $(i,j)$, $(l,j)$ are scalar differential operators
\begin{equation*}
    \begin{aligned}
        L_{i,j}(x,D) =& \sum_{\abs{\alpha_{i,j}}\leq r_{i,j}} a_{\alpha_{i,j}}(x)D^{\alpha_{i,j}},\quad B_{l,j}(x,D) = \sum_{\abs{\beta_{l,j}}\leq q_{l,j}} b_{\beta_{l,j}}(x)D^{\beta_{l,j}},\\
        &i,j\in\set{1,...,M},\text{ }l\in\set{1,...,L},
    \end{aligned}
\end{equation*}
where $\alpha_{i,j},\beta_{l,j}\in \N_0^d$ are multi-indices, $r_{i,j},q_{l,j}\in \N_0$, and $a_{\alpha_{i,j}}:\Omega\to\R$, $b_{\beta_{l,j}}:\partial\Omega\to\R$ are scalar functions. In addition, we have $D^\alpha=\prod_{i=1}^d \partial_{x_i}^{\alpha_i}$ and $\abs{\alpha}= \sum_{i=1}^d \alpha_i$ for a multi-index $\alpha\in\N_0^d$.\\
For $\xi\in\R^d$, we define the \textit{symbols} $L(x,\xi)$ and $B(x,\xi)$ of $L(x,D)$ and $B(x,D)$ by 
\begin{equation*}
    \begin{aligned}
        L_{i,j}(x,\xi) =& \sum_{\abs{\alpha_{i,j}}\leq r_{i,j}} a_{\alpha_{i,j}}(x)\xi^{\alpha_{i,j}},\quad B_{l,j}(x,\xi) = \sum_{\abs{\beta_{l,j}}\leq q_{l,j}} b_{\beta_{l,j}}(x)\xi^{\beta_{l,j}},\\
        &i,j\in\set{1,...,M},\text{ }l\in\set{1,...,L},
    \end{aligned}
    \textbf{ } 
\end{equation*}
with $\xi^\alpha=\prod_{i=1}^d \xi_i^{\alpha_i}$ for a multi-index $\alpha\in\N_0^d$.
\begin{defn}\label{adn-elliptic}
    The system \eqref{elliptic-system} is called \textit{ADN-elliptic} if there exist $s,t\in\Z^M$ such that the following holds:
    \begin{enumerate}
        \item[($\alpha$)] $\deg L_{i,j}(x,\xi) \leq s_i+t_j$;
        \item[($\beta$)] $L_{i,j}(x,\xi)\equiv 0$ if $s_i+t_j<0$;
        \item[($\gamma$)] $\det L^p(x,\xi)\neq0$ for all $\xi\in\R^d\backslash\set{0}$, where $L^p$ is the \textit{principal part} of $L$ defined by
        \begin{equation*}
            L^p_{i,j}(x,D) = \sum_{\abs{\alpha_{i,j}}= s_i+t_j} a_{\alpha_{i,j}}(x)D^{\alpha_{i,j}}, \quad i,j\in\set{1,...,M}.
        \end{equation*}
    \end{enumerate}
    Moreover, $L$ is called a \textit{uniformly elliptic} operator of order $2m$, $m\in\N$, if there exists a constant $C>0$, independent of $x\in\Omega$, such that
    \begin{equation*}
        C^{-1}\abs{\xi}^{2m}\leq \abs{\det L^p(x,\xi)}\leq C\abs{\xi}^{2m} \quad\forall \xi\in\R^d,\text{ }x\in\Omega.
    \end{equation*}
\end{defn}
For a well-posed boundary value problem, it is necessary that $L=m$, which we will assume in the following. In two dimensions, we also need the following definition.
\begin{defn}
    The operator $L$ fulfills the \textit{supplementary condition} or is called \textit{regular elliptic} if, for all linearly independent vectors $\xi,\xi'\in \R^d$, among the roots of the polynomial $\C\ni\sigma\mapsto\det L^p(x,\xi+\sigma\xi')$ there are exactly $m$ with a positive imaginary part.
\end{defn}
Let us now focus on the boundary operator $B$. In the fashion of Definition \ref{adn-elliptic}, we introduce an additional weight $r\in\Z^L$ that needs to satisfy 
\begin{equation*}
    \deg B_{l,j}(x,\xi) \leq r_l+t_j,
\end{equation*}
with the convention that $B_{l,j}(x,\xi)\equiv 0$ if $r_l+t_j<0$. Analogously, we define the principal part $B^p$ of $B$ by
\begin{equation*}
    B^p_{l,j}(x,D) = \sum_{\abs{\beta_{l,j}}= r_l+t_j} b_{\beta_{l,j}}(x)D^{\beta_{l,j}},\quad l\in\set{1,...,L},\text{ }j\in\set{1,...,M}.
\end{equation*}
As a remark, we want to point out that there may be multiple valid choices for the weights $r,s,t$.\\

For the well-posedness of the boundary value problem \eqref{elliptic-system}, it is necessary that the boundary operator $B$ matches the elliptic operator $L$ in some way. In the following, we state a sufficient (and even equivalent) algebraic condition on the principal parts $L^p$ and $B^p$, called \textit{complementing (Lopatinskii--Shapiro) condition}, but first let us introduce some notation. Fix a point $x\in\partial\Omega$. Let $n$ be the unit normal vector at $x$, $\sigma_k^+(x,\xi)$ be the $m$ roots of $\det L^p(x,\xi+\sigma n)$ with positive imaginary part, introduce the polynomial 
\begin{equation*}
    M^+(x,\xi,\sigma)=\prod_{k=1}^m(\sigma-\sigma_k^+(x,\xi)),
\end{equation*}
and, lastly, define $L'$ as the adjugate matrix of $L^p$. If $L^p$ is invertible, the adjugate matrix is given by $L'=(\det L^p)(L^p)^{-1}$.
\begin{defn}\label{complenenting-condition}
    The operators $L$ and $B$ fulfill the \textit{complementing condition} if for every point $x\in\partial\Omega$ and every real nonzero vector $\xi$ orthogonal to $n(x)$ the following holds:
    The rows of the complex matrix-valued polynomial $\C\ni\sigma\mapsto B^p(x,\xi+\sigma n)L'(x,\xi+\sigma n)$ are linearly independent modulo $M^+(x,\xi,\sigma)$, i.e.,
    \begin{equation}
        \sum_{l=1}^mC_l\left(\sum_{j=1}^MB_{l,j}^p(x,\xi+\sigma n)L'_{j,k}(x,\xi+\sigma n)\right)\equiv0 \text{ mod }M^+(x,\xi,\sigma) \text{ for all }k\in\set{1,...,M},
    \end{equation}
    if and only if $C_1=\cdots=C_m=0$.
\end{defn}
\begin{defn}\label{elliptic-problem}
    The boundary value problem \eqref{elliptic-system} is called \textit{elliptic} if:
\begin{equation*}
    \begin{aligned}
        &\text{(i) $L$ is ADN-elliptic;}\quad\quad \text{ (ii) $L$ is uniformly elliptic;}\\
        &\text{(iii) $L$ is regular elliptic;}\quad \text{(iv) $L$ and $B$ satisfy the complementing condition.}
    \end{aligned}
\end{equation*}
\end{defn}

For $q\geq 0$, we set the following product spaces
\begin{equation*}
    X_q=\prod_{j=1}^MH^{q+t_j}(\Omega),\quad Y_q=\prod_{i=1}^MH^{q-s_i}(\Omega),\quad B_q=\prod_{l=1}^mH^{q-r_l-1/2}(\partial\Omega).
\end{equation*}
A key result of ADN theory is the subsequent a priori estimate for elliptic boundary value problems. 
\begin{thm}\label{elliptic-theorem}
    Set $t'=\max t_j$ and $r'=\max(0,\max r_l+1)$. Let $q\geq r'$ and $\Omega\subset\R^d$ be a bounded domain with $C^{q+t'}$ boundary. Moreover, assume that
    \begin{equation*}
        a_{\alpha_{i,j}}\in C^{q-s_i}(\overline{\Omega}),\quad b_{\beta_{l,j}}\in C^{q-r_l}(\partial\Omega) ,\textbf{ } i,j\in\set{1,...,M},\text{ }l\in\set{1,...,L}.
    \end{equation*} If the boundary value problem \eqref{elliptic-system} is elliptic with $f\in Y_q$ and $g\in B_q$, then there exists, for every solution $u\in X_{q}$, a constant $C>0$, independent of $u$, $f$, and $g$, such that
        \begin{equation}\label{elliptic-estimate}
             \sum_{j=1}^M \norm{u_j}_{H^{q+t_j}(\Omega)}\leq C\left(\sum_{i=1}^M \norm{f_i}_{H^{q-s_i}(\Omega)}+\sum_{l=1}^m \norm{g_l}_{H^{q-r_l-1/2}(\partial\Omega)}+ \norm{u}_{L^2(\Omega)}\right).
        \end{equation}
\end{thm}
As a final remark, the term $\norm{u}_{L^2(\Omega)}$ in \eqref{elliptic-estimate} can be omitted if \eqref{elliptic-system} has a unique solution.


\begin{thebibliography}{10}
\bibitem{agmon}S.~\textsc{Agmon}, A.~\textsc{Douglis}, L.~\textsc{Nirenberg}. \newblock{Estimates near the boundary for solutions of elliptic partial differential equations satisfying general boundary conditions II.} Comm. Pure Appl. Math., 17:35--92, 1964.
\bibitem{bardos}C.~W.~\textsc{Bardos}, T.~T.~\textsc{Nguyen}, T.~T.~\textsc{Nguyen}, E.~S.~\textsc{Titi}. \newblock{The inviscid limit for the 2D Navier--Stokes equations in bounded domains.} Kinet. Relat. Models, 15(3):317--340, 2022.
\bibitem{bochev}P.~B.~\textsc{Bochev}, M.~D.~\textsc{Gunzburger}. \newblock{Least-Squares Finite Element Methods.} Applied Mathematical Sciences, vol. 166. Springer, New York, NY, 2009.
\bibitem{boyer}F.~\textsc{Boyer}. \newblock{Trace theorems and spatial continuity properties for the solutions of the transport equation.} Differential and integral equations 18, no 8, 891--934, 2005.
\bibitem{aubin-lions-cite}
X.~\textsc{Chen}, A.~\textsc{Jüngel}, J.~\textsc{Liu}. \newblock{A note on Aubin-Lions-Dubinski\u{i} lemmas}, Acta Appl. Math., 133:33--45, 2014. 
\bibitem{chemetov}N.~V.~\textsc{Chemetov}, F.~\textsc{Cipriano}. \newblock{The inviscid limit for the Navier–Stokes equations with slip condition on permeable walls.} J. Nonlinear Sci., 23(5):731–750, 2013.

\bibitem{clopeau}T.~\textsc{Clopeau}, A.~\textsc{Mikeli\'{c}}, R.~\textsc{Robert}. \newblock{On the vanishing viscosity limit for the 2D incompressible Navier--Stokes equations with the friction type boundary conditions}, Nonlinearity 11,  1625--1636, 1998.
\bibitem{constantin}P.~\textsc{Constantin}, V.~\textsc{Vicol}. \newblock{Remarks on high Reynolds numbers hydrodynamics and the inviscid limit.} J. Nonlinear Sci., 28(2):711--724, 2018.
\bibitem{diperna-lions}R. J.~\textsc{DiPerna}, P.-L.~\textsc{Lions}.
\newblock{Ordinary differential equations, transport theory and Sobolev spaces},
Invent. Math., 98(3):511--547, 1989.
\bibitem{evans}L.~\textsc{Evans}. \newblock{Partial Differential Equations, 2nd ed.} Graduate Studies in Mathematics, vol. 19, American Mathematical Society, Providence, RI, 2010.
\bibitem{gilbarg}
D.~\textsc{Gilbarg}, N.~S.~\textsc{Trudinger}. \newblock{Elliptic Partial Differential Equations of Second Order.} Grundlehren Math. Wiss. 224, Springer-Verlag, Berlin, 1977.
\bibitem{grisvard}P.~\textsc{Grisvard}. \newblock{Elliptic Problems in Nonsmooth Domains.} Classics in Applied Mathematics, vol. 69, Society for Industrial and Applied Mathematics, Philadelphia, 2011.
\bibitem{kato}T.~\textsc{Kato}. \newblock{Remarks on zero viscosity limit for nonstationary Navier--Stokes flows with boundary.} In Seminar on
 nonlinear partial differential equations (Berkeley, Calif., 1983), volume 2 of Math. Sci. Res. Inst. Publ., Springer, New York, 85--98, 1984.
\bibitem{kelliher} J.~P.~\textsc{Kelliher}.
\newblock{Navier--Stokes equations with Navier boundary conditions for a bounded domain in the plane}, SIAM J. Math. Anal., 38, pp. 210--232, 2006. 
\bibitem{kelliher2} J.~P.~\textsc{Kelliher}. \newblock{Observations on the vanishing viscosity limit.} Trans Am Math Soc. 369.3: 2003--2027, 2017.
\bibitem{lauga}E.~\textsc{Lauga}, M.~\textsc{Brenner}, H.~\textsc{Stone}. \newblock{Microfluidics: The no-slip boundary condition.} In C. Tropea,
A. Yarin, and J. F. Foss, editors, Springer Handbook of Experimental Fluid Mechanics, pages 1219--1240. Springer, 2007.
\bibitem{navier}C.~L.~M.~H.~\textsc{Navier}. \newblock{M\'{e}moire sur les lois du mouvement des fluides.} M\'{e}moires de l’Acad\'{e}mie royale des sciences de l’Institut de France, 6:389–440, 1827
\bibitem{nussenzveig}
M.~C.~\textsc{Lopes Filho}, H.~J.~\textsc{Nussenzveig Lopes}, G.~\textsc{Planas}.
\newblock{On the inviscid limit for 2D
incompressible flow with Navier friction condition}, SIAM J. Math. Anal., 36, pp. 1130--1141, 2005.
\bibitem{maekawa}Y.~\textsc{Maekawa}, A.~\textsc{Mazzucato}. \newblock{The inviscid limit and boundary layers for Navier-Stokes flows.} Handbook of Mathematical
Analysis in Mechanics of Viscous Fluids, 781--828, 2018.
\bibitem{masmoudi}N.~\textsc{Masmoudi}. \newblock{Remarks about the inviscid limit of the Navier--Stokes system.} Comm. Math. Phys., 270(3):777--788, 2007.
\bibitem{pedlosky}J.~\textsc{Pedlosky}. \newblock{Geophysical fluid dynamics.} 2nd ed. (Study ed.). New York etc.: Springer-Verlag. XIV, 710 p.; DM 89.00, 1987.
\bibitem{sammartino1} M.~\textsc{Sammartino}, R.~\textsc{Caflisch}. \newblock{Zero viscosity limit for analytic solutions of the Navier-Stokes equation
 on a half-space. I. Existence for Euler and Prandtl equations.} Comm. Math. Phys., 192(2):433--461, 1998.
 \bibitem{sammartino2} M.~\textsc{Sammartino}, R.~\textsc{Caflisch}. \newblock{Zero viscosity limit for analytic solutions of the Navier-Stokes equation
 on a half-space. II. Construction of the Navier--Stokes Solution.} Comm. Math. Phys., 192(2):463--491, 1998.
\bibitem{wilke} Y.~\textsc{Shao}, G.~\textsc{Simonett}, M.~\textsc{Wilke}. \newblock{The Navier--Stokes equations on manifolds with boundary.} J. Differential Equations 416, 1602--1659, 2025.
\bibitem{wiedemann}
C.~\textsc{Seis}, E.~\textsc{Wiedemann}, J.~\textsc{Wo\'{z}nicki}. \newblock{Strong convergence of vorticities in the 2d viscosity limit on a bounded domain.} J. Nonlinear Sci. 36, Article 22, 2026. 
\bibitem{temam}R.~\textsc{Temam}. \newblock{Navier--Stokes Equations. Theory and Numerical Analysis.} Reprint of the 1984 edition, AMS Chelsea, Providence, RI, 2001.
\bibitem{wang}X.~\textsc{Wang}. \newblock{A Kato type theorem on zero viscosity limit of Navier--Stokes flows.} Indiana Univ.
Math. J., 50(Special Issue):223--241, 2001. Dedicated to Professors Ciprian Foias and Roger Temam (Bloomington, IN, 2000).
\end{thebibliography}
\end{document}